\documentclass[11pt]{amsart}
\usepackage{graphicx}
\usepackage{amsmath,amsthm}
\usepackage{amssymb}
\usepackage{pinlabel}
\usepackage[left = 3.5 cm, right = 3.5 cm , top = 3 cm , bottom = 3cm ]{geometry}
\usepackage[utf8]{inputenc}
\usepackage{cite}
\usepackage{todonotes}
\usepackage{wrapfig}
\usepackage{enumerate}

\newtheorem{thm}{Theorem}[section]
\newtheorem{prop}[thm]{Proposition}
\newtheorem{lem}[thm]{Lemma}
\newtheorem{cor}[thm]{Corollary}
\newtheorem{conj}[thm]{Conjecture}
\newtheorem{obs}[thm]{Observation}

\newtheorem{notation}[thm]{Notation}

\newtheorem{operation}[thm]{Operation}

\newcommand{\Text}[1]{\text{\textnormal{#1}}}
\newcommand{\mF}{\mathcal{F}}
\newcommand{\mA}{\mathcal{A}}

\newcommand{\mFone}{\mathcal{F}_1}
\newcommand{\BR}{\mathbb{R}}

\theoremstyle{remark}
\newtheorem{remark}[thm]{Remark}

\theoremstyle{definition}
\newtheorem{definition}[thm]{Definition}

\begin{document}

\title{The fully marked surface theorem}
\author{David Gabai, Mehdi Yazdi}
\date{August 15 , 2020}
\thanks{Partially supported by NSF Grants DMS-1006553 and DMS-1607374.}
\maketitle

\begin{abstract}
In his seminal 1976 paper Bill Thurston observed that a closed leaf S of a foliation has Euler characteristic equal, up to sign, to the Euler class of the foliation evaluated on [S], the homology class represented by S.  The main result of this paper is a converse for taut foliations: if the Euler class of a taut foliation $\mathcal{F}$ evaluated on [S] equals up to sign the Euler characteristic of S and the underlying manifold is hyperbolic, then there exists another \emph{taut} foliation $\mathcal{F'}$ such that $S$ is homologous to a union of leaves and such that the plane field of $\mathcal{F'}$ is homotopic to that of $\mathcal{F}$. In particular, $\mathcal{F}$ and $\mathcal{F'}$ have the same Euler class.

In the same paper Thurston proved that taut foliations on closed hyperbolic 3--manifolds have Euler class of norm at most one, and conjectured that, conversely, any integral cohomology class with norm equal to one is the Euler class of a taut foliation. This is the second of two papers that together give a negative answer to Thurston's conjecture. In the first paper, counterexamples were constructed assuming the main result of this paper.
\end{abstract}

\section{introduction}

In his seminal 1976 paper Bill Thurston observed \cite{thurston1986norm} that if $S$ is a compact leaf of a foliation $\mathcal{F}$, then up to sign, the Euler characteristic of $S$ is equal to the Euler class of $\mathcal{F}$ evaluated on $[S]$ and used this to prove his celebrated result that compact leaves of taut foliations are Thurston \emph{norm-minimizing}.  In brief, let $M$ be a closed, orientable and irreducible 3-manifold and $\mathcal{F}$ be a taut foliation on $M$. By Roussarie--Thurston general position \cite{MR1732868} \cite{roussarie1974plongements} any oriented, embedded, incompressible surface $S$ in $M$ can be isotoped so that any component of $S$ is either a leaf or is transverse to $\mathcal{F}$  except for finitely many (Morse type) saddle singularities. An application of the Poincar\'{e}--Hopf index formula implies that the number of such singularities is exactly the absolute value of the Euler characteristic of $S$. We call a singularity \emph{positive (negative)} if at the point of tangency, the transverse orientation of the surface and the foliation agree (disagree).  Thurston's insight was that the Euler class evaluated on $-[S]$ equals the sum of the signs of these singularities.  In what follows all surfaces are incompressible, a requirement only relevant for tori and annuli; it is automatic for other connected fully marked surfaces in tautly foliated manifolds. A compact surface $S$ is \emph{positive (negative) fully marked} if every component of $S$ is either a leaf whose transverse orientation agrees (disagrees) with the transverse orientation of the foliation, or has only saddle singularities all of which are positive (negative). A surface is \emph{fully marked} if it is positive or negative fully marked. Let $e(\mathcal{F}) \in H^2(M; \mathbb{R})$ denote the Euler class of the tangent bundle of $\mathcal{F}$ and consider the pairing $\langle \hspace{1mm}, \hspace{1mm} \rangle$ between the second cohomology and homology of $M$. A compact surface $S$ is \emph{algebraically fully marked} if 
\[\langle e(\mathcal{F}) , [S] \rangle = \pm \chi(S).  \]
A fully marked surface with boundary should have each boundary component either transverse to the foliation or be a leaf. By the Roussarie--Thurston theorem, in a tautly foliated 3-manifold an
algebraically fully marked surface is isotopic to a fully marked surface. 

Note compact leaves of $\mathcal{F}$ (if it has any) are fully marked; and similarly a finite union of oriented, compact leaves is fully marked provided that the elements of this union (leaves) are oriented consistently. The converse, however, is not true since $\mathcal{F}$ might have no compact leaves while having fully marked surfaces. Indeed, any taut foliation of a hyperbolic 3-manifold can be perturbed to one without compact leaves, without changing the homotopy class of the plane field and hence the Euler class.   The main result of this paper gives a converse to Thurston's theorem for closed hyperbolic 3-manifolds, up to homotopy of the plane fields of the foliations. 

\begin{thm} 
Let $M$ be a closed hyperbolic 3-manifold, $\mathcal{F}$ be a taut foliation on $M$, and $S$ be an algebraically fully marked surface in $M$. There exists a surface $S'$ homologous to $S$ and a taut foliation $\mathcal{G}$ such that 
\begin{enumerate}
\item $S'$ is a union of leaves of $\mathcal{G}$.
\item The oriented plane fields tangent to $\mathcal{F}$ and $\mathcal{G}$ are homotopic.
\end{enumerate}
\label{fullymarked}
\end{thm}

\begin{cor}
Let $M$ be a closed hyperbolic 3-manifold, $\mathcal{F}$ be a taut foliation on $M$, and $S$ be an algebraically fully marked surface in $M$. Assume that $S$ is the unique norm-minimizing surface in its homology class, up to isotopy. There exists a taut foliation $\mathcal{G}$ such that $S$ is a union of leaves and the oriented plane fields tangent to $\mathcal{F}$ and $\mathcal{G}$ are homotopic.
\label{fullymarked2}
\end{cor}

Now assume that $M$ has positive first Betti number and is atoroidal.  Thurston proved that for any taut foliation $\mathcal{F}$ on $M$, the Euler class $e(\mathcal{F})$ has norm at most one and satisfies \emph{the parity condition}. This means that for any $S$, the following inequality holds:
\begin{equation} 
\langle e(\mathcal{F}), [S] \rangle \leq |\chi(S)|, 
\label{inequality}
\end{equation}
and the numbers $\langle e(\mathcal{F}), [S] \rangle$ and $\chi(S)$ have the same parity. The Euler class has norm exactly equal to one if there exists a surface $S$ such that the equality occurs. In particular if $\mathcal{F}$ has some compact leaf, then the norm of the Euler class is equal to one. Thurston conjectured that, conversely, given any integral cohomology class $a \in H^2(M; \mathbb{Z})$ of norm equal to one, there exists a [taut] foliation on $M$ whose Euler class is equal to $a$ \footnote{Presumably Thurston meant to state the conjecture for cohomology with real coefficients rather than integral coefficients, as the general flow of his writing and his subsequent comments about the motivation for the conjecture suggests. See Pages 137--138 in \cite{thurston1986norm}. Thurston did not mention the parity condition in his conjecture. However, it easily follows from the index sum formula, which was known to him, that the parity condition is necessary.}. In \cite{firstpaper} the second author constructed counterexamples to this conjecture assuming Corollary \ref{fullymarked2}. That together with our main result yields our main application.

\begin{thm} There are infinitely many closed hyperbolic 3-manifolds $M$ for which Thurston's Euler class one conjecture does not hold; i.e. there
is an integral point in the unit dual ball, satisfying the necessary parity condition, which is not realized by any taut foliation. \end{thm}

This result resolves the last of three fundamental conjectures offered by Thurston in \cite{thurston1986norm}.  In 1985, positive solutions to the first two were given by the first author in \cite{gabai1983}.  He also proved a partial positive result for Thurston's third conjecture.  

\begin{thm}[Gabai] Let $M$ be a compact oriented irreducible 3-manifold, possibly with toral boundary, and let $a \in H^2(M, \partial M ; \mathbb{R})$ be a vertex of the dual unit ball. Then there is a taut foliation on $M$ whose Euler class is equal to $a$.
\label{thm:Gabai}
\end{thm}
We expect Theorem \ref{fullymarked} to fail in general without allowing to change $S$ within its homology class.

\begin{conj} There exists a closed hyperbolic 3-manifold $M$ supporting a taut foliation $\mathcal{F}$ with a fully marked surface $S$, such that there exists no taut foliation $\mathcal{G}$ on $M$ with oriented plane field homotopic to $\mathcal{F}$ such that $S$ is a union of leaves of $\mathcal{G}$.
\label{generalcase}
\end{conj}

\vskip 8pt
Here we give an informal sketch of the proof of Theorem \ref{fullymarked}. By Roussarie--Thurston general position, the surface $S$ can be isotoped such that each of its components becomes either a leaf or such that the induced singular foliation on the component has only saddle singularities. We do not touch any component of $S$ that is already a leaf. Take the union of components of $S$ for which the second scenario happens, and by abuse of notation call it $S$. 

Note that, there might be two-dimensional Reeb components on $S$. Since $S$ is fully marked, all saddle singularities on $S$ have the same sign, that is the oriented normal vectors to the surface and to the foliation always agree or always disagree. Without loss of generality, we may assume the orientations always agree. We fix a line field that is transverse to both $\mathcal{F}$ and $S$.

Cut $M$ along $S$ to get the manifold $M \setminus \setminus S$. The boundary of $M \setminus \setminus S$ consists of two copies of $S$.  We want to modify the foliation along $S$ by adding leaves to the boundary of $M \setminus \setminus S$ to obtain a foliation on $M \setminus \setminus S$ that is tangential to the boundary. Then the desired foliation $\mathcal{G}$ can be obtained by gluing two copies of $S$ in the boundary of $M \setminus \setminus S$.  Starting with a line field defined in a neighborhood $U$ of $S$ and transverse to $\mathcal{F}$, the modification of $\mathcal{F}$ to $\mathcal{G}$ is supported in $U$ and the leaves of $\mathcal{G}$ continue to be transverse to this line field. It follows that the plane field of $\mathcal{G}$ is homotopic to that of $\mathcal{F}$.  

There are two main technical issues to carry out the proof as stated. First, $\mathcal{G}$ might have Reeb components. The issue comes from certain \emph{bad solid tori} inside the induced foliation on $M \setminus \setminus S$. We show that one can avoid this unpleasant situation by replacing $S$ with a new surface $S'$ with $[S]=[S']$. This is done by defining a set of moves for changing the surface $S$ while preserving its homology class. Moreover, a complexity function is defined, strictly decreasing under these moves that terminates after finitely many moves. At this point no bad solid tori remain.   Second, the extension might require filling in an $A\times I$ where $A$ is an annulus, $A\times 0\subset L$, $L$ a leaf of $\mathcal{G'}$, $A\times 1\subset S$, and $\partial A\times I$ is transverse to $\mathcal{G'}$.  Here $\mathcal{G'}$ is the partially extended foliation.  The problem is that the holonomy on the two sides may not match and hence there is no way to fill in.  Such a problem was encountered in \cite{MR1162560}.  Both these two technical issues require a more global modification of the foliation.  Nevertheless, by using a transverse line field to $\mathcal{F}$, we can modify to $\mathcal{G}$ without changing the homotopy class.

\subsection{Outline}
The paper is organized as follows.  In Section 2, we review the background material. In Section 3, Thurston's theorem on compact leaves of taut foliations is stated.  We note that its proof shows that fully marked surfaces are norm-minimizing and each non torus or annulus component is incompressible. We then state the first author's converse to Thurston's theorem, and show how he used it to give a partial positive solution to Thurston's Euler class one conjecture. In Section 4 \emph{bad solid tori} are introduced and it is shown that at the cost of repeatedly replacing $S$ by a surface in its homology class, and modifying the foliation preserving the homotopy class of its plane field, all the bad solid tori can be eliminated. In Section 5, it has been shown that in the absence of bad solid tori, one can find a \emph{complete system of coherent transversals}. The complete system of coherent transversals is used in Section 7 to ensure that the tautness property is preserved. In Section 6, combinatorial lemmas on train tracks and surfaces are presented to be used in Section 7. In Section 7, various constructions for modifying foliations are presented.  These modifications are shown to preserve the homotopy class of the plane field of the original foliation. Section 7 ends with the proof of Theorem \ref{fullymarked}. In Section 8 we offer a further conjecture.

\section{Background}

\subsection{Taut foliations} 
By a \emph{foliation} of a 3-manifold $M$, we mean a decomposition of $M$ into injectively immersed surfaces that locally looks like the product foliation $\mathbb{R}^2 \times \mathbb{R}$. A \emph{leaf} of the foliation is a connected component of the surfaces in the foliation. Throughout this paper we assume that $M$ is orientable and all foliations are \emph{transversely orientable}, meaning that there is a consistent choice of transverse orientation for the leaves. 

A foliation $\mathcal{F}$ on the compact manifold $M$, transverse to the possibly empty $\partial M$ is called \emph{taut} if every leaf has a \emph{closed transversal}. A closed transversal is a closed loop transverse to the foliation. For taut foliations transverse to $\partial M$, a single transversal suffices (See Page 155 of \cite{calegari2007foliations}). 

\subsection{Regularity of foliations}
A foliation $\mathcal{F}$ is called $C^0$, or \emph{topological}, if the holonomy of its leaves is continuous. It is called $C^{\infty,0}$ if the leaves are smoothly immersed with continuous holonomy. By Calegari, every topological foliation of a 3-manifold is topologically isotopic to a $C^{\infty,0}$ foliation \cite{calegari2001leafwise}. 

\subsection{Suspension foliations}
The exposition here is taken from Chapter V of \cite{camacho2013geometric}. Let $p \colon E \longrightarrow B$ be a fiber bundle with base $B$,  fiber $F$, and total space $E$. We say a foliation $\mathcal{F}$ of $E$ is \emph{transverse to the fibers} if 
\begin{enumerate}
\item Each leaf $L$ of $\mathcal{F}$ is transverse to the fibers and $\dim(L) + \dim(F)= \dim E $.
\item For each leaf $L$ of $\mathcal{F}$, the restriction map $p \colon L \longrightarrow B$ is a covering map.
\end{enumerate}
When the fiber $F$ is compact, Condition (2) is implied by Condition (1), as shown by Ehresmann. 

Given a fiber bundle and a foliation transverse to the fibers, one can construct a representation 
\[ \phi \colon \pi_1(B, b_0) \longrightarrow \text{Homeo}(F), \hspace{3mm} b_0 \in B, \]
that is the holonomy around the based loops lying in $B$. Conversely:

\begin{thm}
Let $B$ and $F$ be connected manifolds. Given a representation 
\[ \phi \colon \pi_1(B, b_0) \longrightarrow \text{Homeo}(F), \hspace{3mm} b_0 \in B, \]
one can construct a fiber bundle $E(\phi)$ over the base $B$ and with fiber $F$, and a foliation $\mathcal{F}(\phi)$ transverse to the fibers of $E(\phi)$ such that the holonomy of $\mathcal{F}(\phi)$ is equal to $\phi$.
\end{thm}

We are mainly interested in the case that $F = [0,1]$ or $S^1$ is one-dimensional, and the image of $\phi$ lies in $\text{Homeo}^+(F)$, that is the group of orientation-preserving homeomorphisms of $F$. The construction is as follows: 

Denote by $\tilde{B}$ the universal cover of $B$. Consider the action of $\pi_1(B,b_0)$ on $\tilde{B} \times F$ defined as 
\[ \gamma \in \pi_1(B , b_0), \hspace{3mm} (\tilde{b} , f ) \in \tilde{B} \times F; \hspace{3mm} \gamma \cdot (\tilde{b} , f) := (\gamma \cdot \tilde{b} \hspace{1mm} , \hspace{1mm} \phi(\gamma). f),\]
where the action on the first factor is by covering transformations. This action preserves the product foliation on $\tilde{B} \times F$, meaning that it sends leaves to (possibly different) leaves. Hence there is an induced foliation on the quotient 
\[ E(\phi) : = (\tilde{B} \times F) \setminus \pi_1(B,b_0), \]
that satisfies the desired properties.

\subsection{Corners}
Consider a codimension-one foliation of a 3-manifold $M$ with non-empty boundary. Let $p \in \partial M$ be a point. We say that $p$ is a \emph{tangential point} if there is a foliated neighborhood of $p$ in $M$ that is homeomorphic to a foliated neighborhood of $(0,0,0)$ in 
\[ \{ (x,y,z) \hspace{1mm} | \hspace{1mm} x, y \in \mathbb{R}, \hspace{2mm}z \geq 0 \},  \]
where the leaves consist of the planes $z = $ constant. By definition, $p$ is a \emph{transverse point} if there is a foliated neighborhood of $p$ in $M$ that is homeomorphic to a foliated neighborhood of $(0,0,0)$ in the foliation of 
\[ \{(x,y,z) \hspace{1mm} | \hspace{1mm} y,z \in \mathbb{R}, \hspace{2mm} x \geq 0\}, \]
where the leaves are the half-planes $z=$ constant.

We say $p$ is a \emph{convex corner} if there is a foliated neighborhood of $p$ that is homeomorphic to a foliated neighborhood of $(0,0,0)$ in the foliation of 
\[ \{(x,y,z) \hspace{1mm} | \hspace{1mm} x \geq 0 , \hspace{2mm}y \in \mathbb{R}, \hspace{2mm} z \geq 0\}, \]
where the leaves consist of the half-planes $z = $ constant. 

A point $p$ is a \emph{concave corner}, if there is a foliated neighborhood of $p$ that is homeomorphic to a foliated neighborhood of  $(0,0,0)$ in the foliation of 
\[ \{  (x,y,z) \hspace{1mm} | \hspace{1mm} y \in \mathbb{R}, \text{ and} \hspace{2mm} \big( x \geq 0  \text{ or } z \geq 0 \big) \}, \]
where the leaves consist of the planes and half-planes $z = $constant.

A point $p$ is a \emph{node} if there is a foliated neighborhood of $p$ that is homeomorphic to a foliated neighborhood of  $(0,0,0)$ in the foliation of 
\[  \{(x,y,z) \hspace{1mm} | \hspace{1mm} x \geq 0 , \hspace{1mm}y \in \mathbb{R}, \hspace{1mm} z \geq 0\} \cup  \{  (x,y,z) \hspace{1mm} | \hspace{1mm} z<0, \text{ and} \hspace{2mm} \big( x \geq 0  \text{ or } y \geq 0 \big)\},  \]
by the half-planes and three-quarter-planes $z$ = constant. See Figure \ref{corners}, where the dashed lines indicate that the leaves are cut open in the figure in order to make them more visible. 

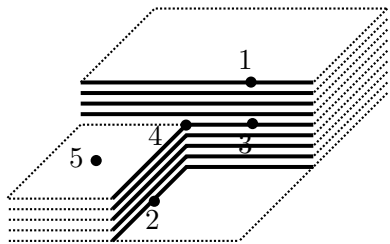
\begin{figure}
\labellist
\pinlabel $1$ at 85 65 
\pinlabel $3$ at 85 35
\pinlabel $4$ at 53 38
\pinlabel $5$ at 25 30
\pinlabel $2$ at 52 8
\endlabellist
\centering
\includegraphics[width= 2 in]{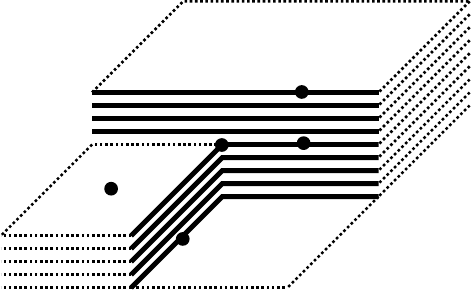}
\caption{The points labelled with $1, 2, 3, 4$ and $5$ respectively are convex corner, concave corner, transverse point, node, and tangential point.}
\label{corners}
\end{figure}

Define the \emph{tangential boundary} of $M$, $\partial_{\tau}M$, as the closure of the union of tangential points in $\partial M$. Let the \emph{transverse boundary} of $M$, $\partial_\pitchfork M$, be the closure of the complement of the tangential boundary, $\partial_{\pitchfork} M : = \overline{\partial M - \partial_{\tau}M}$. In particular, convex corners, concave corners, and nodes are included in both $\partial_{\pitchfork}M$ and $\partial_\tau M$. 

\subsection{$I$-bundle replacement} Let $L$ be a leaf of the transversely orientable codimension-one foliation $\mathcal{F}$ of the compact 3-manifold $M$. Informally speaking, we want to blow air into the foliation along $L$ and fill the gap with a packet of leaves.  This well-known and frequently used operation, also known as \emph{Denjoy blow-up}, goes back to 1932 \cite{denjoy1932courbes} when Denjoy showed how to replace a leaf by a product bundle of leaves.  Here we replace $L$ by any foliated product $I$-bundle over $L$.  As an example, two dimensions lower, we have the Denjoy blow-up of the dyadic rationals in the interval.  This replaces each dyadic rational in $(0,1)$ by a closed interval.  The reverse operation starts with the standard middle thirds Cantor set in the interval.  Pass back to the interval by identifying the closure of each complementary arc to a point.   

In our setting we start with a transverse line field $\mathcal{V}$ to $\mathcal{F}$.  We now describe the most interesting case, which is when $L$ is dense in $M$.  Blowing up $L$  produces a laminated space $X\subset M$ which is transversely a Cantor set.  Here $\mathcal{V}$ induces the product structure $L\times I$ on $M\setminus X$ completed with the induced path metric.  We recover $M$ by identifying each connected interval of $L\times I$ to a point.  $I$-\emph{bundle  replacement} is the operation of passing from $\mathcal{F} $ to $\mathcal{G}$ by filling in $L\times I$ with a foliated bundle transverse to the $I$-factor.  For more details, see for example, \cite[Example 4.14]{calegari2007foliations}.  If the starting foliation is taut, the new constructed foliation remains taut.  Indeed, the same single curve, transverse to all the leaves of $\mathcal{F}$ is a transversal for all the leaves of $\mathcal{G}$.  

We will need a slight variation of the $I$-bundle replacement, which we call a partial $I$-bundle replacement. Assume that $M$ satisfies the following boundary condition: Let $\partial_{\tau}M$ and $\partial_{\pitchfork}M$ be the tangential and transverse boundary of $M$ respectively. The intersection of $\partial_\tau M$ and $\partial_{\pitchfork} M$ is a finite union of disjointly embedded $1$-complexes and simple closed curves in $\partial M$, whose vertices (respectively edges and simple closed curves) correspond to nodes (respectively convex and concave corners) on $\partial M$. 

Let $T$ be a component of $\partial_{\tau}M$, and $L$ be the leaf of $\mathcal{F}$ containing $T$. Define the preferred side of $L$ as the side facing $\partial M$. Let $R$ be a compact subsurface of $L \cap (\partial_\tau M)$. Define a \textbf{partial $I$-bundle replacement} along $L-\text{int}(R)$ as the result of first doing an $I$-bundle replacement along $L$ on the preferred side where $I$ is identified with $[0,1]$, and then removing the restriction of the $(0,1]$-bundle over $\text{int}(R)$. If $L_0$ is a connected component of $L - \text{int}(R)$, we define the partial $I$-bundle replacement along $L_0$ by first doing a partial $I$-bundle replacement along $L - \text{int}(R)$ and then collapsing the $I$ fibers above all other components of $L - \text{int}(R)$. See Figure \ref{partial-I-bundle-replacement} for a schematic picture of partial $I$-bundle replacement in one dimension lower.

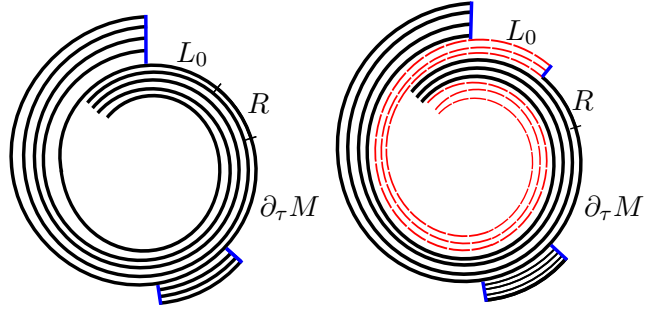
\begin{figure}
	\labellist
	\pinlabel $R$ at 102 83
	\pinlabel $R$ at 235 84
	\pinlabel $L_0$ at 75 103
	\pinlabel $L_0$ at 210 112
	\pinlabel $\partial_\tau M$ at 115 40
	\pinlabel $\partial_\tau M$ at 248 40
	\endlabellist
	
	\centering
	\includegraphics[width= 3 in]{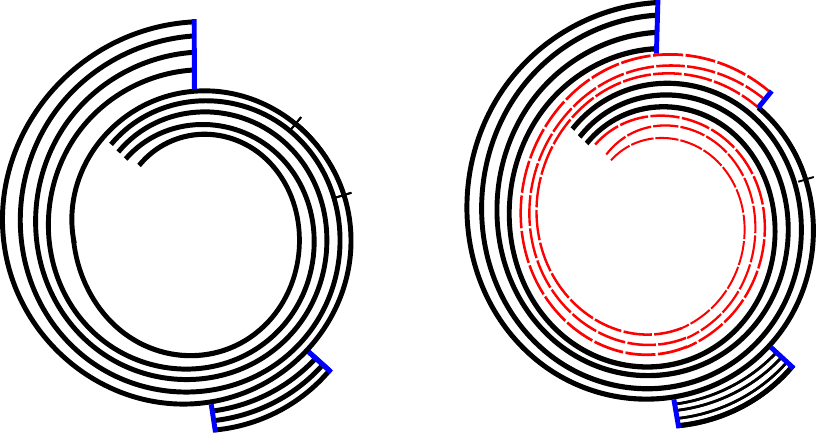}
	\caption{A schematic picture of partial $I$-bundle replacement in one dimension lower: the left (respectively) right hand side show the picture before (respectively after) the partial $I$-bundle replacement. The vertical segments (in blue) indicate the transverse part of the boundary. The dashed lines (in red) indicate the newly added leaves, i.e. the $I$-bundle over $L_0$. Note that a new component of $\partial_\pitchfork M$, next to $(\partial R) \cap L_0$, is created.}
	\label{partial-I-bundle-replacement}
\end{figure}

The following will be used for establishing tautness of newly constructed foliations.

\begin{obs}
Let $M$ be a compact 3-manifold, and $\mathcal{F}$ be a codimension-one foliation on $M$. Let $\mathcal{F}'$ be obtained from $\mathcal{F}$ by an $I$-bundle replacement. Then any transversal (respectively transverse vector field) for $\mathcal{F}$ is naturally a transversal (respectively transverse vector field) for $\mathcal{F}'$. Moreover, if $\gamma$ is a transversal for $L$, a leaf of $\mathcal{F}$ that is blown up, then $\gamma$ is a transversal for the blown up leaves. In particular, if $\mathcal{F}$ is taut, then so is $\mathcal{F}'$.
\label{I-bundle-replacement}
\end{obs}

\subsection{Homeomorphisms of the interval}
In this section, we gather some of the results that will be needed about orientation-preserving homeomorphisms of the interval. 

\begin{lem}
If $F$ is any surface with boundary which is not compact planar and $b$ is a boundary component of $F$, then there are foliations of $F \times I$ ($I$ is a closed interval), transverse to the $I$ factor that have a given holonomy on $b$ and trivial holonomy on all other boundary components. In the remaining case that $F$ is compact planar (not a disk), if $b$ and $b'$ are two boundary components with the induced orientations from $F$, then there exists a foliation transverse to the $I$ factor that has a given holonomy $\mu$ on $b$ and $\mu ^{-1}$ on $b'$ and trivial holonomy on all other boundary components \cite{MR1162560}. 
\label{transversefoliations}
\end{lem}
The next lemma is a modification of Lemma 2.1 in \cite{MR1162560}.

\begin{lem}
Suppose $u,v$ are given orientation-preserving homeomorphisms of the interval. There exist an orientation-preserving homeomorphism $\tau$ such that $\tau$ is conjugate to:

a) $u \cdot \tau ^{-1} \cdot v$ \hspace{10mm} b) $u \cdot \tau \cdot v$ \hspace{10mm} c) $u \cdot \tau$ 

d) $\tau \cdot v$ \hspace{19mm} e) $u \cdot \tau^{-1}$ \hspace{11mm} f) $\tau^{-1} \cdot v$\\
Here $u \cdot v$ denotes concatenation, likewise for $u \cdot \tau ^{-1} \cdot v$.
\label{concatenation}
\end{lem}

\begin{proof} 
Identify the interval with $[-1,1]$. By the concatenation $f$ of $u$ and $v$ we mean that there exist $-1 < a < 1$ such that $f_{|[-1,a]}$ is conjugate to $u$ and $f_{|[a,1]}$ is conjugate to $v$. The choice of $a$ does not affect the conjugacy class of $f$. Note that the inverse of $u \cdot v$ is equal to $u^{-1} \cdot v^{-1}$.

\textit{a)} Break this interval into symmetric pieces as     
\[   [-1, - \frac{1}{2} ] , [- \frac{1}{2}, - \frac{1}{3} ], \cdots , [\frac{1}{3},  \frac{1}{2} ] , [ \frac{1}{2}, 1 ].      \]
Define $\tau$ to be conjugate to $u$ and $v$ respectively on $[-1, - \frac{1}{2} ]$ and $[ \frac{1}{2}, 1 ]$. Then define it to be conjugate to $u^{-1}$ and $v^{-1}$ respectively on $[- \frac{1}{2}, - \frac{1}{3} ]$ and $[\frac{1}{3},  \frac{1}{2} ]$, and continue so on. Finally set $\tau (0) = 0$. As constructed we have $\tau = u \cdot u^{-1} \cdot \cdots \cdot v^{-1} \cdot v$, and therefore its inverse is $u^{-1} \cdot u \cdot \cdots \cdot v \cdot v^{-1}$. Hence $\tau$ is conjugate to $u \cdot \tau^{-1} \cdot v$.

\textit{b)} Similar to the previous part.

\textit{c)} Break the interval as the following:
\[ [-1,0] , [0,\frac{1}{2}],[\frac{1}{2},\frac{2}{3}], \cdots \]
On each subinterval, define $\tau$ to be conjugate to $u$. Finally set $\tau(1)=1$. Then we have $\tau = u \cdot u\cdot \cdots$, which is conjugate to $u \cdot \tau$.

\textit{d,e,f)} Similar to part \textit{c}.
\end{proof}
\subsection{Roussarie--Thurston general position}

Let $M$ be a closed, orientable 3-manifold and $\mathcal{F}$ be a taut foliation on $M$. Roussarie \cite{roussarie1974plongements} and Thurston \cite{thurston1972foliations} proved that any connected, embedded, incompressible surface $S \subset M$ can be isotoped such that $S$ is either a leaf or is transverse to $\mathcal{F}$ except at finitely many points of saddle tangencies. 

The theorem holds in more generality when $M$ has boundary; in this case we assume that $\mathcal{F}$ is transverse to $\partial M$, and each component of $\partial S$ is either transverse to $\mathcal{F}|\partial M$ or tangent to $\mathcal{F}|\partial M$. Both Roussarie and Thurston state the result for connected surfaces in transversely $C^2$-foliations.  With foliations now known to be at least $C^{\infty,0}$ the proof holds for all taut foliations.   The proof works for disconnected surfaces as well, the key point being that a surface tangent to a  compact leaf with nontrivial holonomy can be isotoped slightly to be a fully marked surface that is not a leaf.  In fact, unless $\mathcal{F}$ is a bundle, we can arrange that after the isotopy no component of $S$ is a leaf. The first author generalized it to the case of \textit{immersed} incompressible surfaces, and without any orientability assumption on the manifold and the foliation, only assuming that the foliation is tangentially smooth \cite{gabai2000combinatorial}.  

\subsection{Haefliger's theorem on compact leaves}

\begin{thm}(Haefliger \cite{haefliger1962varietes})
Let $\mathcal{F}$ be a codimension-one foliation of a compact $n$-manifold $M$. The union of compact leaves of $\mathcal{F}$ is a compact subset of $M$. Moreover if $\mathcal{F}$ is transversely orientable and $K$ is a compact $(n-1)$-dimensional manifold, the union of leaves of diffeomorphism type $K$ is compact as well.
\end{thm}

By a \emph{packet of leaves}, we mean either 

\begin{enumerate}
\item an injectively immersed copy of $K \times [0,1]$, together with a foliation that is transverse to the interval factor, or 
\item a single leaf $K$.
\end{enumerate}
In the first case, $K \times \{0 \}$ and $K \times \{ 1 \}$ are called the \emph{end leaves of the packet}, while in the second case the single leaf $K$ is considered as the end leaf. Note the induced foliation on the packet does not have to be the product foliation. It follows from Haefliger's theorem that \emph{when $\mathcal{F}$ is transversely orientable and $K$ is compact, the union of leaves of $\mathcal{F}$ that are diffeomorphic to $K$ appear in finitely many packets}. To see this, define an equivalence relation on the leaves of $\mathcal{F}$ that are diffeomorphic to the compact $K$ as follows:
\[ K_1 \sim K_2   \hspace{5mm} \text{if} \hspace{5mm} \text{there is a packet of leaves, whose end leaves are } K_1 \text{ and } K_2.  \]
It is easy to see that this defines an equivalence relation, and there are finitely many equivalence classes.

\subsection{Poincar\'{e}--Hopf index theorem}  The classical Poincar\'{e}--Hopf index formula asserts that if $X$ is a vector field on a closed manifold $N$, then the Euler characteristic of $N$ is equal to the alternating sum of the indices.  In this paper we use a special case of a generalization due to Goodman \cite{goodman1975closed} which states that if $N$ is a compact 3-manifold with a non-vanishing vector field that points in (respectively out, respectively tangent) along $A$ (respectively $B$, respectively $T$) where $A\cup B\cup T=\partial N$, then $\chi(A)=\chi(B)$.  In application, $\mathcal{F}$ is a foliation on $N$ where $A$ (respectively $B$) is union of the components of $\partial N$ consisting of leaves of $\mathcal{F}$  where the normal orientation  points in (respectively out) and $T$ is the union of (torus) components of $\partial N$ transverse to leaves of $\mathcal{F}$.

\subsection{Embedded tori in closed hyperbolic 3-manifolds}
The following is standard. 

\begin{lem}
Let $M$ be an irreducible 3-manifold and let $T\subset \text{int}(M)$ be a compressible embedded torus. Then, either $T$ bounds a solid torus inside $M$ or $T$ is contained in a 3-ball. In particular if $T$ contains a curve homotopically essential in $M$, then $T$ bounds a solid torus. 
\end{lem}

\begin{proof}
Since $T$ is compressible, some simple closed curve in $T$ bounds an essential embedded disk $D\subset M$ with $D\cap T=\partial D$. Surger $T$ along $D$ to obtain an embedded 2-sphere $S$. Since $M$ is irreducible, $S$ bounds a 3-ball $B$. There are two cases. If $D\cap B=\emptyset$ then $T$ bounds a solid torus while if $D\subset B$, then $T \subset B$.
\end{proof}

\begin{cor}
Let $N$ be a closed hyperbolic 3-manifold and $T\subset N$ be an embedded torus. If $T$ contains a curve homotopically essential in $M$, then $T$ bounds a solid torus.
\label{torus}
\end{cor}
\begin{proof}
Hyperbolic 3-manifolds are irreducible and any torus in a closed hyperbolic 3-manifold is compressible. Therefore, the previous lemma applies.
\end{proof}

\section{Thurston's Euler class one conjecture}
Roussarie \cite{roussarie1974plongements} and Thurston \cite{thurston1972foliations} realized that taut foliations and embedded incompressible surfaces have an `efficient intersection property', and furthermore Thurston deduced inequality (\ref{inequality}) from that general position \cite{thurston1986norm}. Thurston introduced a natural norm on second homology of 3-manifolds, now called the \emph{Thurston norm} and studied connections between taut foliations and this norm. Putting inequality (\ref{inequality}) in the language of the Thurston norm, he obtained that the Euler class of any taut foliation of a 3-manifold has dual Thurston norm at most one.

\begin{thm}[Thurston]Let M be a compact oriented 3-manifold and $\mathcal{F}$ a codimension-one, transversely oriented foliation of M. Suppose that $\mathcal{F}$ contains no Reeb components and each component of $\partial M$ is either a leaf of $\mathcal{F}$ or a surface T such that $\mathcal{F}$ is transverse to T and each leaf of $\mathcal{F}$ which intersects $T$ also intersects a closed transverse curve (e.g. $\mathcal{F}|T$ has no Reeb components). Then 
\[ x^*(e)\leq 1 \]
holds in 
\begin{enumerate}
	\item $H^2(M)$.
	\item $H^2(M, \partial M)$.
\end{enumerate} 
Here $x^*$ is the dual Thurston norm and $e$ is the Euler class of the tangent plane bundle to the foliation $\mathcal{F}$. 
\label{thurston}
\end{thm}

The idea of the proof  of Theorem \ref{thurston} is the following \cite{thurston1986norm}. Let $a \in H_2(M , \partial M ; \mathbb{Z})$ be an integral homology class. In order to compute the quantity $\langle e(\mathcal{F}), [S] \rangle$, first represent $a$ by an incompressible surface $S$. Putting $S$ in Roussarie--Thurston general position, we may assume that $S$ is transverse to $\mathcal{F}$ except at finitely many points of saddle, or \emph{circle tangencies}. In fact circle tangencies can be avoided whenever the foliation is taut as shown by Thurston \cite{thurston1972foliations}, so we do not discuss them here (In any case their contribution to $\langle e(\mathcal{F}), [S] \rangle$ is zero even if they exist). Assign $-1$ (respectively $+1$) to a saddle tangency $p \in S$, if the oriented normal vectors to $S$ and $\mathcal{F}$ agree (respectively disagree) at the point $p$. Then the quantity $\langle e(\mathcal{F}), [S] \rangle$ can be obtained by adding up all the numbers corresponding to saddle tangencies. Now by the Poincar\'{e}--Hopf formula, the number of saddle tangencies is equal to $|\chi(S)|$. It is clear that the sum of the numbers associated to saddle tangencies is maximum in absolute value, when all of the numbers are equal so there is no cancellation. Hence we have 
\[ |\langle e(\mathcal{F}), [S] \rangle| \leq |\chi(S)|. \]

The proof of Thurston's theorem implies the following important property of fully marked surfaces: any fully marked surface is norm-minimizing.

\begin{cor}
Let $M \neq S^2 \times S^1$ be a closed orientable 3-manifold, and $\mathcal{F}$ be a taut foliation on $M$. Any fully marked surface $S$ has no sphere component, and is norm-minimizing and incompressible. Moreover, any compact leaf of the induced foliation on $S$ is $\pi_1$-injective in $S$, and hence in $M$.
\label{incompressible}
\end{cor}

Theorem \ref{thurston} shows that not every integral second cohomology class can be realized as the Euler class of some taut foliation. In particular, the number of such classes is finite if $M$ is an-annular and atoroidal. This is in contrast with the case of general foliations on closed 3-manifolds, where Wood showed that every integral second cohomology class satisfying the parity condition can be realized as the Euler class of some transversely oriented foliation \cite{wood1969foliations} (see \cite{firstpaper}). Conversely Thurston conjectured the following (see \cite{thurston1986norm}, Page 129, Conjecture 3). 

\begin{conj}[Thurston]If M has no `essential' singular tori, and if $a \in H^2(M;\mathbb{Z})$
is any element with $x^*(a) = 1$, then there is some [taut] foliation $\mathcal{F}$ of M such that $e(\mathcal{F}) = a$.
\label{conj}
\end{conj}

Thurston showed that the unit ball for the dual Thurston norm is a convex polyhedron whose vertices are integral points \cite{thurston1986norm}. Later the first author proved the conjecture holds for the vertices of the dual ball (See \cite{gabai1997problems}, Page 24, Remark 7.3). 

\newtheorem*{thm:Gabai}{Theorem \ref{thm:Gabai}}
\begin{thm:Gabai}
[Gabai]
Let $M$ be a compact oriented irreducible 3-manifold, possibly with toral boundary, and let $a \in H^2(M, \partial M; \mathbb{R})$ be a vertex of the dual unit ball. Then there is a taut foliation on $M$ whose Euler class is equal to $a$.
\end{thm:Gabai}

\begin{proof}
Let $a \in H^2(M, \partial M)$ be a vertex of the dual unit ball, and $\mathcal{C}$ be the face dual to the point $a$. As $a$ is a vertex, $\mathcal{C}$ is a top-dimensional face. Let $\bar{a} \in H_2(M, \partial M)$ be a rational point in $\mathcal{C}$, and $S$ be an embedded norm-minimizing surface representing a multiple of the homology class $\bar{a}$. Denote by \\
\[ \langle \hspace{1mm}  , \hspace{1mm} \rangle : H^2(M,\partial M) \times H_2(M, \partial M) \longrightarrow \mathbb{R}, \]
the pairing between the second cohomology and homology of $M$. By definition we have
\[ \langle a , \bar{a} \rangle = 1 = x(\bar{a}). \]
By \cite{gabai1983}, there exists a taut foliation $\mathcal{F}$ on $M$ such that $S$ is a leaf of $\mathcal{F}$, the foliation $\mathcal{F}$ is transverse to $\partial M$, and $\partial \mathcal{F}$ has no Reeb component. We show that $e(\mathcal{F})=a$.\\

Since $\bar{a}$ is in the interior of the top-dimensional face $\mathcal{C}$, we can choose a basis $\bar{a}_1 , \bar{a}_2 , \cdots , \bar{a}_n$ for the second homology of $M$ such that each $\bar{a}_i$ lies in the closure of $\mathcal{C}$ and
\[ \bar{a}=t_1 \bar{a}_1 + \cdots + t_n \bar{a}_n ,\]
with $0<t_i<1$ and $\sum_{i=1}^{n} t_i =1$. By hypothesis, for each $1 \leq i \leq n$ we have 
\[ \langle a , \bar{a}_i \rangle = 1 = x(\bar{a}_i). \]
Since $S$ is a leaf of $\mathcal{F}$
\[ |\langle e(\mathcal{F}) , [S] \rangle| = |\chi(S)| \]
\[ \implies |\langle e(\mathcal{F}) , \bar{a} \rangle| = 1 .\]
Therefore we have\\
\[ 1 = |\langle e(\mathcal{F}),\bar{a} \rangle|=|\sum_i t_i \hspace{1mm} \langle e(\mathcal{F}) , \bar{a}_i \rangle| \underbrace{ \leq }_{(1)} \sum_i t_i |\hspace{1mm} \langle e(\mathcal{F}) , \bar{a}_i \rangle| \underbrace{\leq}_{(2)} \sum t_i \hspace{1mm}x(\bar{a}_i)= \sum_i t_i = 1.\]
Here the implication (1) is the triangle inequality, and (2) is the fact that the Euler class $e$ has dual norm at most one. So each of the inequalities in (2) should be in fact equality. As $t_i > 0$, for each index $i$ we have\\
\[ \langle e(\mathcal{F}),\bar{a}_i \rangle = x(\bar{a}_i)=\langle a,\bar{a}_i \rangle .\]
Since $\bar{a}_i$ are a basis for the second homology of $M$, we have $e(\mathcal{F})=a   $.
\end{proof}

\section{Eliminating Bad Solid Tori}

In section \S 7 we will give a procedure that starts with a fully marked surface $S$ in the 3-manifold $M$ with a taut foliation $\mathcal{F}$ and produces a new foliation $\mathcal{G}$ with $S$ a union of leaves, such that $\mathcal{F}$ and $\mathcal{G}$ have homotopic plane fields.  In general $\mathcal{G}$ will have Reeb components, hence will not be taut.  The problem is that the pair $(\mathcal{F},S)$ may have \emph{bad solid tori}.  In this section we show that after  replacing $\mathcal{F}$ with $\mathcal{F}'$ by $I$-bundle replacement and replacing $S$ with a homologous surface $S'$, then $(\mathcal{F}',S')$ has no bad solid tori.  Our $S'$ may have more or less components than $S$, in particular $S$ itself may be disconnected.
\vskip 8pt

\begin{definition}
Let $S$ be a compact surface, and $\mathcal{F}|S$ be a singular foliation on $S$ with only finitely many singular points $P$, all of which are saddles. If $L$ is a leaf of $\mathcal{F}|S$ passing through a singularity, then a \textbf{separatrix} is a connected component of $L \setminus P$ together with the one or two points of P from which it emanates from.  If no separatrix is compact, then we say that $F$ has the \textbf{compact-free separatrix property}.
\end{definition}

The following lemma allows for simplification of various technical issues in this paper.
\begin{lem}
Let $M$ be a compact orientable 3-manifold and $\mathcal{F}$ be a taut foliation of $M$ such that every component of $\partial M$ is either tangent or transverse to $\mathcal{F}$. Let $S$ be an embedded orientable incompressible surface that is transverse to $\mathcal{F}$ except at finitely many saddles tangencies. One can do I-bundle replacement along some of the leaves of $\mathcal{F}$, together with an arbitrary small isotopy of $S$ to obtain $\mathcal{F}'$ and $S'$ such that $\mathcal{F'}|S'$ has the compact-free separatrix property.
\label{noseparatrix}
\end{lem}

\begin{proof}

Let $p_1, \cdots , p_n$ be the saddle singularities on $S$. By Theorem 7.1.10 from \cite{MR1732868}, we may assume that no two distinct $p_i$ are connected by a separatrix inside $\mathcal{F}|S$. This does not need the hypotheses on $\mathcal{F}$ being taut, or  $S$ being incompressible.

The next step is to get rid of a separatrix from a singularity to itself. Let $\gamma$ be a separatrix from $p_1$ to itself, and $L$ be the leaf of $\mathcal{F}$ containing $\gamma$. The loop $\gamma$ is homotopically nontrivial in $S$; otherwise if $D$ was a disk bounding $\gamma$ in $S$, there would have been a center tangency inside $D$ by the Poincar\'{e}--Hopf formula. 

The surface $S$ is embedded, incompressible and two-sided. Hence $S$ is $\pi_1$-injective. Therefore $\gamma$ is homotopically nontrivial in $M$ and hence in $L$ as well. 

Since $p_1$ is a saddle tangency, there are locally 4 separatrices coming out of it. There are either two compact separatrices passing through $p_1$ or just one. Consider a small standard neighborhood of $p_1$ where the surface $S$ can be seen as the graph of the function $z = x^2- y^2$ and the foliation $\mathcal{F}$ is by horizontal planes. Do an I-bundle replacement along $L$, and call the resulting foliation $\mathcal{F}'$. Note that $\mathcal{F'}|S$ is obtained from $\mathcal{F}|S$ by a singular $I$-bundle replacement along $L \cap S$. We abuse notation by calling the new singularity of $\mathcal{F}'|S$ by $p_1$. Note that if the holonomy along the loop $\gamma$ is a shift, then the number of compact separatrices passing through $p_1$ is reduced, and the status of separatrices passing through $p_j$ for $j>1$ is not changed. Repeat this with other compact separatrices.
\end{proof}
\begin{remark} \label{reeb unchanged} If every Reeb component of $\mathcal{F}|S$ is disjoint from the singularities, then the operation of Lemma \ref{noseparatrix} maintains that property and keeps the number of Reeb components unchanged.  \end{remark}

\begin{notation} Given the closed oriented embedded surface $S\subset M$, let $M_1$ denote $M\setminus S$ compactified with the induced path metric. We sometimes call $M_1$ the closed complement of $M \setminus S$. Given the foliation $\mathcal{F}$ on $M$, then $\mathcal{F}_1$ denotes the induced foliation on $M_1$.   We denote by $S_0$ (respectively $S_1$) the component of $\partial M_1$ where the orientation points in (respectively out) at every point of tangency. At times it will be useful to view $M_1$ as being naturally immersed in $M$.  \end{notation}

\begin{definition}
A \textbf{bad annulus} is a properly embedded annulus leaf $U$ of $\mathcal{F}_1$ with both boundary components lying on $S_1$ or both lying on $S_0$.
\label{defbadannulus}
\end{definition}

\begin{definition} Let $S\subset M$ be a fully marked surface with respect to the taut foliation $\mathcal{F}$.
A \textbf{bad solid torus} for $(\mathcal{F},S)$ is an embedded solid torus $B\subset M_1$ which is bounded by bad annuli $U_i$ together with annuli subsurfaces $A_j$ such that either 
\begin{enumerate}
\item $\forall j, A_j \subset S_1$, and the normal to the foliation $\mathcal{F}$ points out of $B$ along all $U_i$, or
\item $\forall j, A_j \subset S_0$, and the normal to $\mathcal{F}$ points into $B$ along all $U_i$. See Figure \ref{badsolidtorus}.
\end{enumerate}
\label{defbadsolidtorus}
\end{definition}

\begin{figure}
\labellist
\pinlabel $A_1$ at 45 170
\pinlabel $A_2$ at 105 170 
\pinlabel $A_3$ at 180 170
\pinlabel $U_1$ at 75 142
\pinlabel $U_3$ at 110 60
\pinlabel $U_2$ at 140 132
\pinlabel $B$ at 110 105
\endlabellist

\centering
\includegraphics[width=2.5 in]{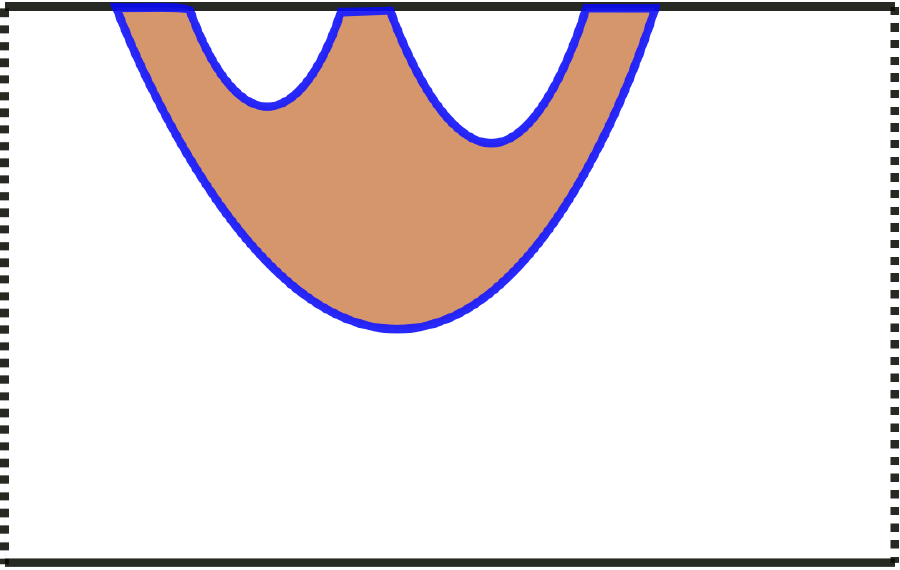}
\caption{A schematic picture for a bad solid torus. The horizontal lines represent $S_0$ and $S_1$.}
\label{badsolidtorus}
\end{figure}
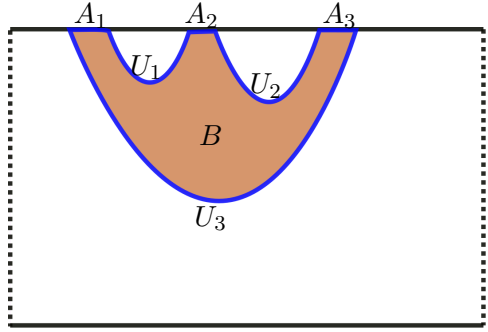

\begin{prop}
Let $M$ be a closed hyperbolic 3-manifold, $\mathcal{F}$ be a taut foliation on $M$ and $S$ be a fully marked surface. One can do $I$-bundle replacement along some of the leaves to obtain $\mathcal{F}'$ and replace $S$ by a homologous fully marked surface $S'$ such that $(\mathcal{F}',S')$ has no bad solid torus and $\mathcal{F'}|S'$ has the compact-free separatrix property. \label{bad annulus}
\end{prop}

\noindent By Lemma \ref{noseparatrix} we can assume that $\mathcal{F}|S$ has no arc connecting separatrices.  In what follows the reader should remember that this implies that $\mathcal{F}|S$ contains no topologically immersed circle that intersects a singularity.  Before embarking on the proof we give some preliminary definitions and lemmas.

\begin{lem}
Let $M$ be a compact orientable 3-manifold, $\mathcal{F}$ be a taut foliation on $M$ and $S$ be a fully marked surface. Let $T = A\cup B$ be a torus, where $A$ and $B$ are annuli, $A$ is in a leaf and $B$ is in S and $\mathcal{F}|S$ is a suspension. Then $T$ is a $\pi_1$-injective torus.
\label{incompressible torus}
\end{lem}

\begin{proof}
The torus $T$ contains a smooth circle leaf $\alpha$ of $\mathcal{F}|S$. The curve $\alpha$ is $\pi_1$-injective in $M$, by Corollary \ref{incompressible}.  Since $\mathcal{F}$ is transversely oriented, the transverse orientation on $B$ agrees with the induced transverse orientation on the suspension $\mathcal{F}|A$.  Therefore, $T$ can then be isotoped slightly to be transverse to $\mathcal{F}$ and have $\alpha$ as a leaf of the induced foliation which is a suspension. Every nontrivial element of $\pi_1(T)$ is represented by either a multiple of $\alpha$ or a curve transverse to $\mathcal{F}$. The latter curves are essential in $M$ by Novikov's theorem \cite{novikov1965topology}.  We conclude that $T$ is $\pi_1$-injective in $M$.
\end{proof}

\begin{definition}\ \vspace{1mm}
\begin{itemize}
\item A \textbf{string} is an annulus $S^1 \times [0,1]$ lying inside a leaf of $\mathcal{F}$ such that $[0,1]$ factor can be decomposed as $0=t_0 < t_1 < \cdots < t_k=1 $ where each $S^1 \times [t_i , t_{i+1}]$ intersects $ S$ exactly along its boundary circles.  The \textbf{height} of the string is equal to $k$, and $S^1 \times [t_i,t_{i+1}]$ are called the \textbf{pieces of the string}. 
\item A \textbf{maximal string} is a string that can not be extended to a string of larger height.
\item A \textbf{packet of strings} is an embedding $i: A\times I \longrightarrow M$, where $A = S^1 \times [0,1]$ is an annulus and $[0,1]$ factor can be decomposed as $0 = t_0< t_1 < \cdots < t_k =1$ where the image of each \textbf{piece} $S^1 \times [t_i, t_{i+1}] \times I$ is a packet of leaves in $M_1$ with end leaves $i \big( S^1 \times [t_i, t_{i+1}] \times \partial I \big)$, and intersecting $\partial M_1$ exactly along $i \big( S^1 \times \{ t_i , t_{i+1} \} \times I \big)$. The number $k$ is called the \textbf{height} of the packet. A single string is also considered as a special case of a packet of strings.
\item Define the \textbf{tangential boundary} of a packet of strings as the restriction of $i$ to $ A\times \partial I$, and the \textbf{transverse boundary} as the restriction of $i$ to $\partial A \times I$. 
\end{itemize}
\end{definition}

\begin{remark} A packet of strings intersects $S$ in annuli that are foliated as suspensions.  By abuse of notation, we identify a string with its image in $M$. \end{remark}

\begin{remark}
Let $M$ be a closed hyperbolic 3-manifold, $\mathcal{F}$ be a taut foliation on $M$, and $S$ be a fully marked surface. Then the condition that the transverse boundaries (and hence also tangential boundaries) of a packet of strings are disjoint is automatic by Lemmas \ref{incompressible torus} and \ref{torus}, since $M$ does not contain any incompressible torus. In other words, we only need to assume that the restriction of $i$ to the interior of $A \times I$ is injective.
\end{remark}


\begin{lem}
If $S\subset M$ is a fully marked surface with respect to the taut foliation $\mathcal{F}$ in the closed atoroidal 3-manifold $M$, then the height $k$ of any string $s$ is uniformly bounded by some $K < \infty$.  \label{bounded}
\end{lem}

\begin{proof}  By Haefliger's theorem, there are finitely many packets $P_1 , \cdots, P_m$ of leaves of $\mathcal{F}_1$ in $M_1$ which contain all the annuli leaves of $\mathcal{F}_1$.  If $\text{Height}(s)>m$, then $s$ passes through some packet, say $P_1$, in at least two of its pieces.  Thus there exists an annulus $A$ in the transverse boundary of $P_1$ and an annulus $B\subset s$ whose union is an embedded torus or Klein bottle $T\subset M$.  

If $T$ is a Klein bottle, then the boundary of a regular neighborhood $N(T)$ bounds a solid torus to the outside and hence $M$ is either reducible or a Seifert fibered space.  Now suppose that $T$ is a torus.  By lemma \ref{incompressible torus}, $T$ is $\pi_1$-injective in $M$.  

\end{proof}



\begin{definition} Two maximal strings $s_0$ and $s_1$ are \textbf{equivalent} if there exists a packet of strings $P$ whose tangential boundary is $s_0 \cup s_1$, and moreover any string in $P$ with boundary lying on the transverse boundary of $P$ is also maximal. \label{equivalence-maximal} \end{definition}

\begin{lem} The number of  equivalence classes of maximal strings is finite.\end{lem}

\begin{proof}The proof is by successive use of the pigeonhole principle. By Haefliger's theorem there are finitely many, say $K_0$, packets that include all the annuli leaves of $\mathcal{F}_1$. We call these 1-packets.  One can more or less construct a branched surface by gluing the 1-packets together in the obvious manner. The branch loci are circles that lie in $S$, and at most $K_0$ sheets can branch from such a circle in either direction. Let $K$ be as in Lemma \ref{bounded}. We show that the number of equivalence classes of maximal strings is less than $(2K_0+1)^{2K}$. 

To prove this, assume the contrary; that there are $t = (2K_0+1)^{2K}$ maximal strings, no two of which are in the same equivalence class. Call the mentioned maximal strings $s_1, s_2, \cdots , s_t$. Since there are 
\[  K_0 < 2K_0+1 \] 
1-packets, by the pigeonhole principle there are at least $(2K_0+1)^{2K-1}$ distinct strings $s_i$ that pass through the same 1-packet $P$. Fix an orientation for the $[0,1]$ direction of $P$. Let $C$ be the bottom transverse boundary of $P$ with respect to the chosen orientation. Heading `downward', a string can either end at $C$ or continue into one of the 1-packets. There are at most $2K_0 +1$ possibilities as Figure \ref{finite string} shows. 

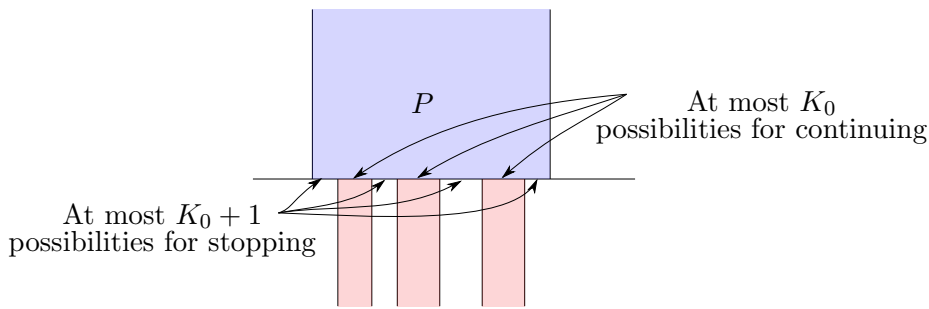
\begin{figure}
\labellist
\pinlabel $P$ at 150 180
\pinlabel {$ \text{At most $K_0$} $} at 450 180
\pinlabel {$ \text{possibilities for continuing} $} at 450 155
\pinlabel {$ \text{At most $K_0+1$} $} at -80 80
\pinlabel {$ \text{possibilities for stopping} $} at -80 55
\endlabellist

\centering
\includegraphics[width=2 in]{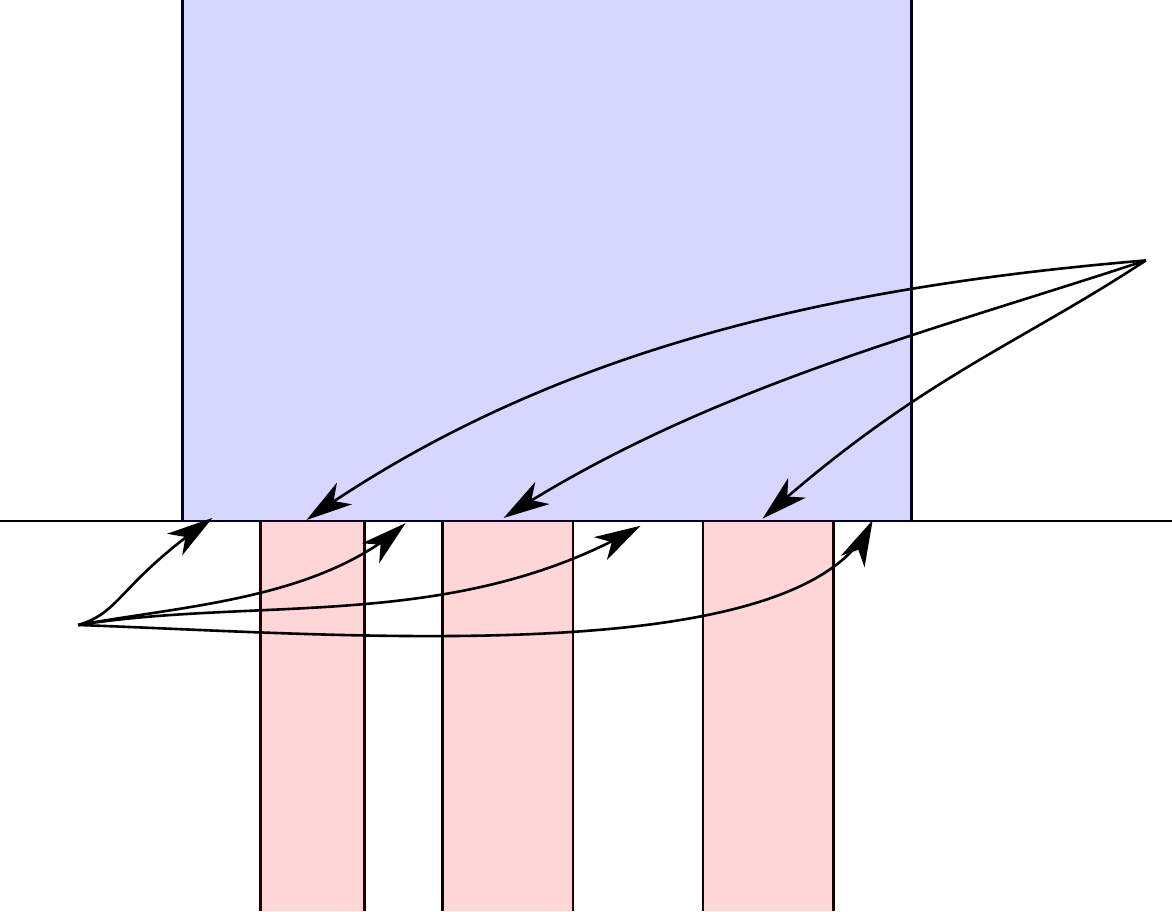}
\caption{The string can either end at $C$ or continue into the 1-packets.}
\label{finite string}
\end{figure}

Here $K_0$ possibilities are counted for continuing into a 1-packet, and $K_0+1$ possibilities are considered for ending at $C$. Therefore at least $(2K_0+1)^{2K-2}$ of the strings $s_i$ have the same fate. If they end at $C$, we start heading upward at $P$ and repeat the same argument. In any case we can only go downward a maximum of $K-1$ steps and similarly for upward, since the length of any string is at most $K$. Note that at this point, still at least $2K_0+1$ strings $s_i$ had the same fate all along, and they have ended on both sides (up and down) together. But this means that they are in the same equivalence class. We came to a contradiction, so the number of equivalence classes is at most $(2K_0+1)^{2K}$.
\end{proof}

\begin{proof}[Proof of Proposition \ref{bad annulus}:]
Define the following complexity function for the pair $(\mathcal{F},S)$:
\[ C = (c_1,c_2), \]
where $c_1$ is the number of Reeb components on $S$, and $c_2$ is sum of the heights of the equivalence classes of maximal strings. Equip $C$ with the lexicographic order, that is
\[ (c_1,c_2) \leq (d_1,d_2) \iff c_1 < d_1 \hspace{4mm} \Text{or} \hspace{4mm} (c_1 =d_1 \hspace{1mm},\hspace{1mm} c_2 < d_2). \]
We show that if there exists a bad solid torus then one can replace $S$ with a new surface $S'$, homologous to $S$, such that the complexity function $C$ for the pair $(\mathcal{F},S')$ is less than the corresponding one for $(\mathcal{F},S)$. Recall that the boundary of a bad solid torus $B$ consists of $m$ bad annuli $U_i$ together with $m$ annuli subsurfaces $A_i$ of $S_1$ (or $S_0$). We may assume that $A_i \subset S_1$, as the other case is similar. \\

First we outline the argument for the case $m=1$. Therefore, $\partial B = U \cup A$. 
\begin{figure}
\labellist
\pinlabel $A$ at 200 100
\pinlabel $U$ at 140 -20
\endlabellist

\centering
\includegraphics[width=2 in]{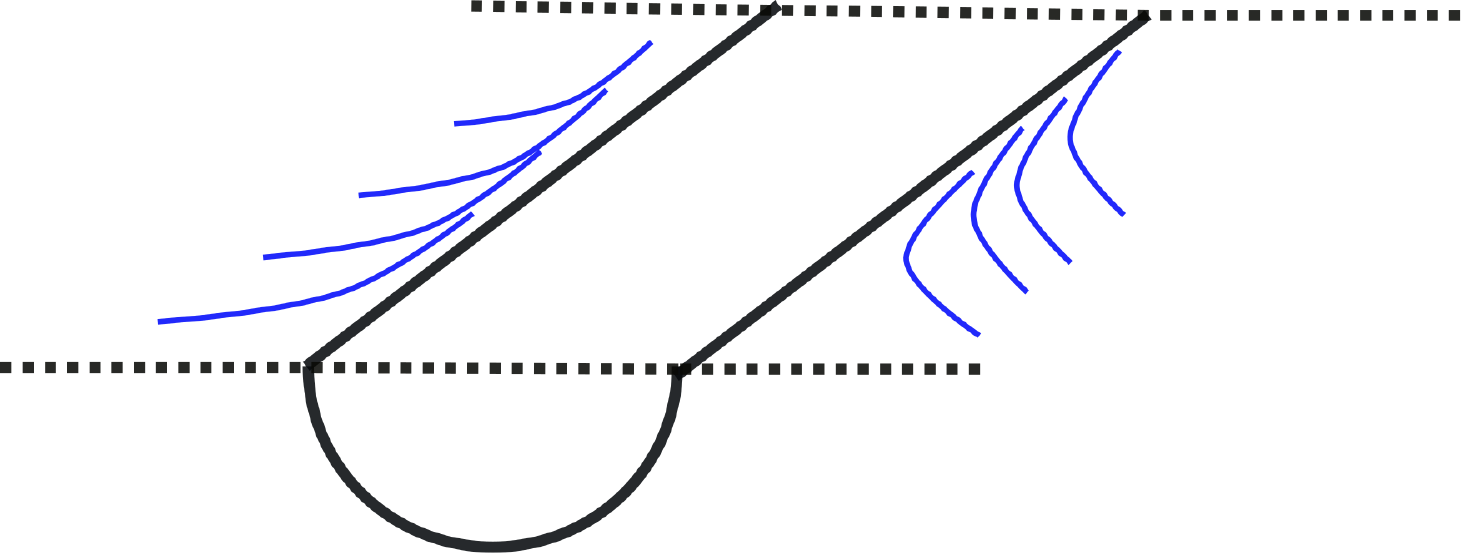}
\caption{A maximal bad solid torus with $m=1$.}
\label{separatingannulus}
\end{figure}

\begin{figure}
\labellist
\pinlabel $U'$ at 80 50
\endlabellist

\centering
\includegraphics[width=1.5 in]{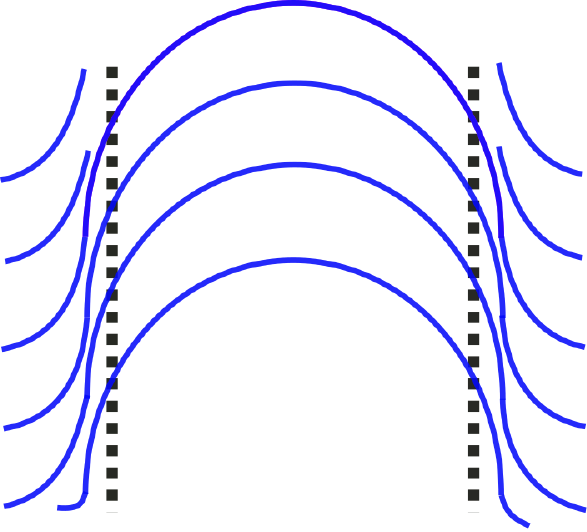}
\caption{Induced foliation on $U'$ (the area between dashed lines)}
\label{inducedonu}
\end{figure}

\noindent Assume that $B$ is maximal in the sense that it can not be extended further along $U$ by adding a packet of strings (Figure \ref{separatingannulus}). This has the following consequence: 

Let $\partial U = b_1 \cup b_2$. Since $b_1$ and $b_2$ are freely homotopic in $U$ they have the same germinal holonomy. Since $B$ is maximal, the germinal holonomy of $b_i$ on the side not contained in $A$ can not have fixed points except for the origin. See Figure \ref{separatingannulus}.\\

Replace $S$ with $S-A+U$ and push it slightly out of $B$ to make it disjoint from $U$. Call this new surface $S'$. The surface $S'$ is homologous to $S$, since $U-A$ bounds the solid torus $B$. We show that the complexity function for the pair $(\mathcal{F},S')$ is less than the one for $(\mathcal{F},S)$. \\

First we examine what happens to the number of Reeb components. Note that there is at least one Reeb component on $A$, since the normal vector to $\mathcal{F}$ points out of $B$ along $U$. After replacing $S$ with $S'$ all the Reeb components on $A$ disappear. We show that at most one new Reeb component can be created, and therefore $c_1$ is non-increasing. 

Let $U'$ be the portion of $S'$ obtained from $U$ after pushing out. The induced foliation on $U'$ is as in Figure \ref{inducedonu}. Let $R$ be a Reeb component of $S'$. If $R \cap U' = \emptyset$ then $R$ is a Reeb component of $S$ as well. If $R \cap U' \neq \emptyset$ then no leaf in $U'$ can  be contained in a boundary leaf of $R$, since every leaf of $U'$ has a closed transversal. Therefore all leaves of $U'$ should be part of the interior leaves of $R$. The only ways that $U'$ can be completed to a Reeb component are as in Figure \ref{twopossible}. In the first (respectively second) scenario, there are closed leaves $b_1' , b_2' \subset S-A$ parallel to $b_1$ and $b_2$ respectively such that the induced foliation on the annulus cobounding $b_i$ and $b_i'$ is a suspension of a shift homeomorphism (two-dimensional Reeb component) for $i=1,2$. In the second scenario $c_1$ decreases by at least two. We now show that if $c_1$ remains constant, then $c_2$ decreases.  The crucial point is that \emph{all} compact leaves of $\mathcal{F}|S'$ are disjoint from $U'$. \\

\begin{figure}
\labellist
\pinlabel $b_1$ at 55 -10
\pinlabel $b_2$ at 175 -10 
\pinlabel $b_2'$ at 230 -10
\pinlabel $b_1'$ at 0 -10
\pinlabel $U$ at 115 30
\endlabellist

\centering
\includegraphics[width=3 in]{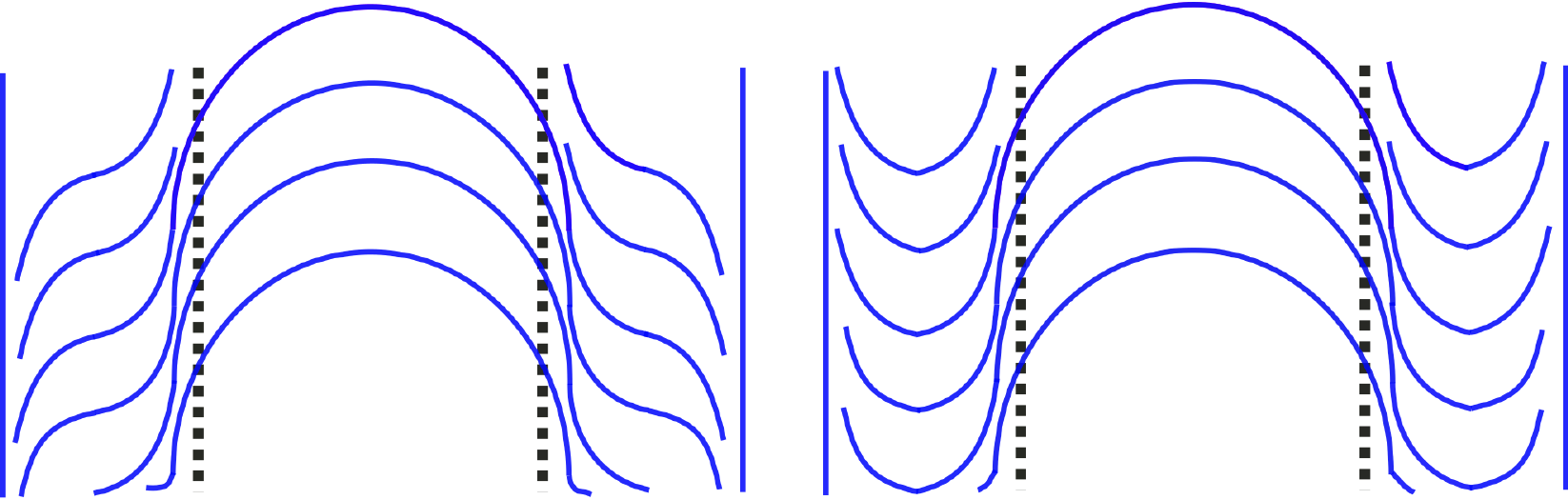}
\caption{Two scenarios for the induced foliation on a neighborhood of $U'$}
\label{twopossible}
\end{figure}
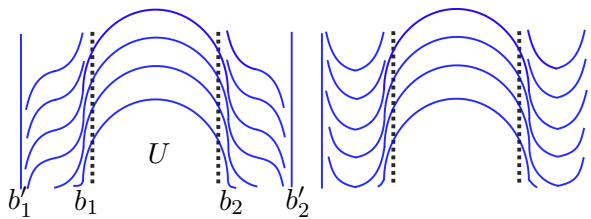
We now define a map 
\[ j:\{ \text{maximal strings for }(\mathcal{F},S) \} / \sim \hspace{2mm} \longrightarrow \hspace{2mm} \{ \text{maximal strings for } (\mathcal{F},S') \} / \sim', \]
where the equivalence relations $\sim$ and $\sim'$ are defined as in Definition \ref{equivalence-maximal} corresponding to the pairs $(\mathcal{F},S)$ and $(\mathcal{F}, S')$ respectively. We prove the following properties of $j$:
\begin{enumerate}
\item $j$ is well-defined.
\item $\text{Height}(j(s)) \leq \text{Height}(s)$. Moreover, there exists at least one string $s$ such that the inequality is strict.
\item $j$ is surjective.
\end{enumerate}
Recall that $c_2$ is defined as sum of the heights of the equivalence classes of maximal strings. Once proven, (1), (2) and (3) together imply that $c_2$ strictly decreases. Note that no injectivity assumption is needed to make the conclusion.\\

First we show that (1) and (2) hold. Let $s$ be a maximal string for $S$. If $s \cap A = \emptyset$ then define $j(s):=s$. Note that there is at least one maximal string $s$ for $S$ such that $s \cap A \neq \emptyset$ (there should be some maximal string including $U$ for example). If $s \cap A \neq \emptyset $ then define $j(s)$ as follows. 

Assume $\Text{Height}(s) = k$ and $0=t_0 < t_1 < \cdots < t_k=1$ be the interval decomposition for $s$; therefore $s$ has $k$ pieces corresponding to the subintervals $[t_i,t_{i+1}]$ for $0 \leq i \leq k-1$. If $s$ hits the annulus $A$ in the first (last) consecutive $r$ moments $t_0 , \cdots , t_{r-1}$ then delete the pieces of $s$ corresponding to the first (last) $r$ subintervals. After doing this, the beginning and the end point of the string lie on $S-A$. Next, if the string hits $A$ at moment $t_j$, remove $t_j$ from the list. This has the effect of joining some of the pieces of the string together. Define $j(s)$ as the new string. 

The string $j(s)$ is maximal, otherwise the string $s$ would not have been maximal either. Note that $\Text{Height}(j(s)) \leq \Text{Height}(s)$ is immediate, and the inequality is strict for at least one maximal string $s$ (for example the maximal string containing $U$). It is also possible that $j(s)$ is the empty string. This finishes the definition of $j$, but we still need to check that the definition does not depend on the choice of representative for the equivalence class.

Suppose that $s_1$ and $s_2$ are equivalent maximal strings for $S$. We need to show that $j(s_1)$ and $j(s_2)$ are equivalent maximal strings for $S'$. Let $J$ be a packet of strings having $s_1 \cup s_1$ as its tangential boundary. Crucially, each component of $J \cap S$ lies completely in $A$ or completely in $S-A$. To see this, let $0=t_0 < t_1 < \cdots < t_k=1$ be the interval decomposition of $J$ with $\text{Height}(J)=k$. There are two cases to consider.  

First, let $0< i < k$. Then the restriction of $J$ to $t= t_i$ can not have non-empty intersections with both $A$ and $S-A$. Otherwise the bad solid torus $B$ would not have been maximal as it could be extended further along $U$ by adding some part of $J$ to it. 

Secondly, The restriction of $J$ to $t= t_0$ (respectively $t=t_k$) can not have non-empty intersections with both $A$ and $S-A$. Otherwise either $B$ would not have been maximal or one of the intermediate strings in $J$ would not have been maximal (could be extended to a longer string by adding $U$). 

This establishes the claim that each component of  $J \cap S$ lies completely in $A$ or completely in $S-A$. Therefore when we look at the pieces of $s_1$ and $s_2$, they go through the same process for defining their image under $j$. That is $s_1$ intersects $A$ in the first consecutive $r$ moments if and only if $s_2$ hits $A$ in the first consecutive $r$ moments, and so on. Furthermore, we can do exactly the same process for the packet $J$ to obtain a packet $J'$ whose tangential boundary is $j(s_1) \cup j(s_2)$. Every intermediate string in $J'$ is maximal, since the same was true for $J$. This shows that $j(s_1)$ and $j(s_2)$ are equivalent, hence $j$ is well-defined.

Now we certify that (3) holds. Let $s'$ be a maximal string for $S'$. Recall that none of the leaves inside $U'$ are part of a closed trajectory. Hence, the boundary of each piece of $s'$ should lie on $S'-U'$, which is the same as $S-A$. Extend $s'$ to a maximal string for $S$ by adding annuli pieces to the beginning and end of it, and also subdividing the pieces if it intersects $A$. Let $s$ be the string obtained this way. Then $s$ is a maximal string for $S$ and $j(s)=s'$. Hence $j$ is surjective.

This completes the proof of (1), (2) and (3), which together imply that $c_2$ decreases when $m=1$.\\

Now consider the general case that there exists a bad solid torus $B$. Let 
\[ \partial B = \bigcup\limits_{i=1}^{m} A_i \cup \bigcup\limits_{i=1}^{m}U_i , \]
where $U_i$ is a bad annulus and $A_i$ is an annulus subsurface of $S_1$. We can assume that $B$ is maximal (can not be extended along any of $U_i$). Each $A_i$ contains at least one Reeb component since the normal vector to $\mathcal{F}$ points out of $B$ along all $U_i$. Let $\hat{S}$ be the surface obtained by pushing
\[ S - \bigcup\limits_{i=1}^{m} A_i + \bigcup\limits_{i=1}^{m} U_i  \]
slightly out of $B$.

Arguing as before if the number of Reeb components on $\hat S$ equals that of $S$, then to each $b_i$ component of $\partial A_i$, there exists a leaf $b'_i$ of $\mathcal{F}|S$ such that $b_i, b'_i$ bound an annulus $C_i$ whose interior is disjoint from $\cup A_i$ and where the induced foliation on $C_i$ is a suspension.  In addition, if $b'_i\in\partial A_j$ for some $j$, then the number of Reeb components on $\hat S$ is also reduced.  Thus the analogue of the first scenario holds for each component of $\partial A_i$. Let $S'$ be the result of deleting the  torus components of $\hat S$.  Since $M$ is atoroidal, all such torus components are homologically trivial, and hence $S'$ is homologous to $\hat S$.  Finally, if $\hat S$ has the same number of Reeb components as $S$, and $\hat S \neq S'$, then $S'$ has fewer Reeb components.   

Now assume that $S'$ has the same number of Reeb components as $S$.  We now repeat the argument for the $m=1$ case to conclude that $c_2$ is reduced.  Again the crucial observation is that \emph{all} compact leaves avoid the modified part of $S'$.\\

To complete the induction step we need to show that, after possibly doing $I$-bundle replacement the conclusion of Lemma \ref{noseparatrix} holds and the complexity has been reduced. 

If the number of Reeb components on $S$ equals that of $S'$, then Lemma \ref{noseparatrix} still holds  since any leaf through the modified part of $S'$ is non-compact.  Actually, by construction, no leaf through the modified part of $S'$ hits a singularity. Therefore, in this case $c_1$ is constant and $c_2$ is strictly reduced. 

If the number of Reeb components drops, then the conclusion of Lemma \ref{noseparatrix} may fail to hold. Before addressing this we observe that the Reeb components of $S'$ are all disjoint from the singularities.  To see this observe that if there is a compact arc $\alpha$ lying in a leaf with both endpoints on singularities, then $\alpha$ passes through the modified part of $S'$ and hence there is a closed transversal through $\alpha$.  It follows that $\alpha$ is disjoint from the Reeb components.  Therefore by Remark \ref{reeb unchanged}, after applying Lemma \ref{noseparatrix}, $S'$ has the same number of Reeb components. In this case, $c_1$ is strictly reduced, and hence the pair $(c_1 , c_2)$ equipped with the lexicographic ordering is strictly reduced, no matter if $c_2$ is increased or not.
\end{proof}

\section{A complete system of coherent transversals}

\begin{definition}
The \textbf{positive orientation} on a fully marked surface $S$, without torus components, is the orientation such that at each point of tangency on $S$, the normal orientations to $S$ and the transverse orientation to the foliation agree. A transverse arc $\gamma$ for $S$ is \textbf{positive} at an intersection point $p \in \gamma \cap S$ if the orientation of $\gamma$ and the positive orientation of $S$ are compatible at $p$.
\end{definition}

The following is the main result of this section. 
\begin{prop} \label{transversals}  Let $\mathcal{F}$ be a taut foliation on the closed hyperbolic 3-manifold $M$, $S$ a positively oriented fully marked surface, $M_1$ the closed complement of $M\setminus S$ and $\mathcal{F}_1$ the induced foliation on $M_1$. Assume that $\mathcal{F}|S$ has the compact-free separatrix property. Then $(\mathcal{F},S)$ has no bad solid tori if and only if there exist finitely many positive closed transversals $\gamma_1, \cdots, \gamma_n$ to $\mathcal{F}$ such that if $\gamma=\cup\gamma_i $ then,\\

i) $\gamma$ intersects every leaf of $\mathcal{F}_1$, and 

ii) every intersection of $\gamma$ with $S$ is transverse and positive.\end{prop}

\begin{definition} A positive system of transversals satisfying i), ii) is called a \textbf{complete system of coherent transversals}. A given positive transversal arc or simple closed curve is called \textbf{coherent} if it satisfies ii).  \end{definition}
	
	In the Section \S 7 we show how to modify $\mathcal{F} $ to $\mathcal{G}$ so that $S$ is a union of compact leaves.  If in addition $(\mathcal{F},S)$ has no bad solid tori, then we  show that $\gamma$ is transverse to $\mathcal{G}$ and intersects each leaf and hence  $\mathcal{G}$ is taut.
	
\begin{proof} Since the boundary of a bad solid torus which intersects $S_1$ (respectively $S_0$) consists of annular leaves of $\mathcal{F}_1$ whose normals point out (respectively in) and annuli subsurfaces of $S_1$ (respectively $S_0$) whose normals point out (respectively in) it follows that the tangential boundary of a bad solid torus cannot coherently intersect a closed transversal.  Thus the absence of bad solid tori is a necessary condition for the existence of a complete system of coherent transversals. 

The proof of the other direction is a modification of an argument of Goodman \cite{goodman1975closed}.  We now assume that $(\mathcal{F},S)$ has no bad solid tori.  Let $L$ be a leaf of $\mathcal{F}_1$.   Define 
 \[ \mathcal{A}_L = \{ q \in M \hspace{2mm}| \hspace{2mm} \Text{there exists a coherent transverse arc  from} \hspace{2mm} L\setminus S \hspace{2mm} \Text{to} \hspace{2mm}  q \} \]

It suffices to show that  $\mathcal{A}_L=M$ for all leaves $L$ of $\mFone$  transverse to $S$, including those disjoint from $S$.  If $L\subset \mathcal{A}_L$ then we can find a coherent closed transversal $\gamma' $ through $L$.  If $L$ is tangent to $S$ and $\mathcal{A}_{L^+}=M$ for some leaf $L^+$ just to the positive side of $L$, then by piecing together a transverse arc from $L^+$ to a leaf $L^-$ just to the negative side of $L$ and a transverse arc from $L^-$ to $L^+$ through $L$ we obtain a closed  transversal $\gamma'$ through $L$.   Now $\gamma'$ serves as a coherent transversal for leaves comprising an open set of $M_1$, so the result follows by compactness.  We now fix a leaf $L$ of $\mFone$ transverse to $S$.  In what follows, we drop the subscript in $\mA_L$.

\vskip 10 pt
We now establish some facts about $\mA$. 

\begin{enumerate}[i)]
	\item $\mathcal{A}$ is open in $M$.
	\item $\mathcal{A} - S$ is saturated by leaves of $\mFone|\text{int}(M_1)$, where $\text{int}(M_1)$ is the interior of $M_1$. 
	\item If $U=\BR^2\times \BR$ is a foliation chart for $\mF$ disjoint from $S$, then $U\cap \mathcal{A}$ is connected, open and saturated.
	\item If $U=\BR^2\times \BR$ is a foliation chart for $\mF$ that intersects $S$ in $0\times \BR^2$ and $q=(0,0,0)\in \bar\mA\setminus \mA$, then $V=U\cap \mathcal{A}$ is open and connected and after possibly passing to a smaller chart, includes quadrants as in Figure \ref{quarters}. 
	\item A point $x$ of tangency between $S$ and $\mF$ does not lie in $\bar\mA\setminus \mA$. Proof: By assumption, the leaf of $\mathcal{F}|S$ containing $x$ is non-compact and hence has an accumulation point. Denote by $L$ the leaf of $\mathcal{F}_1$ containing $x$. So using iii), if some nearby (in $M_1$) leaf to $L$ is in $\mathcal{A}$, then so is $L$.
	\begin{figure}
		\centering
		\includegraphics[width=4 in]{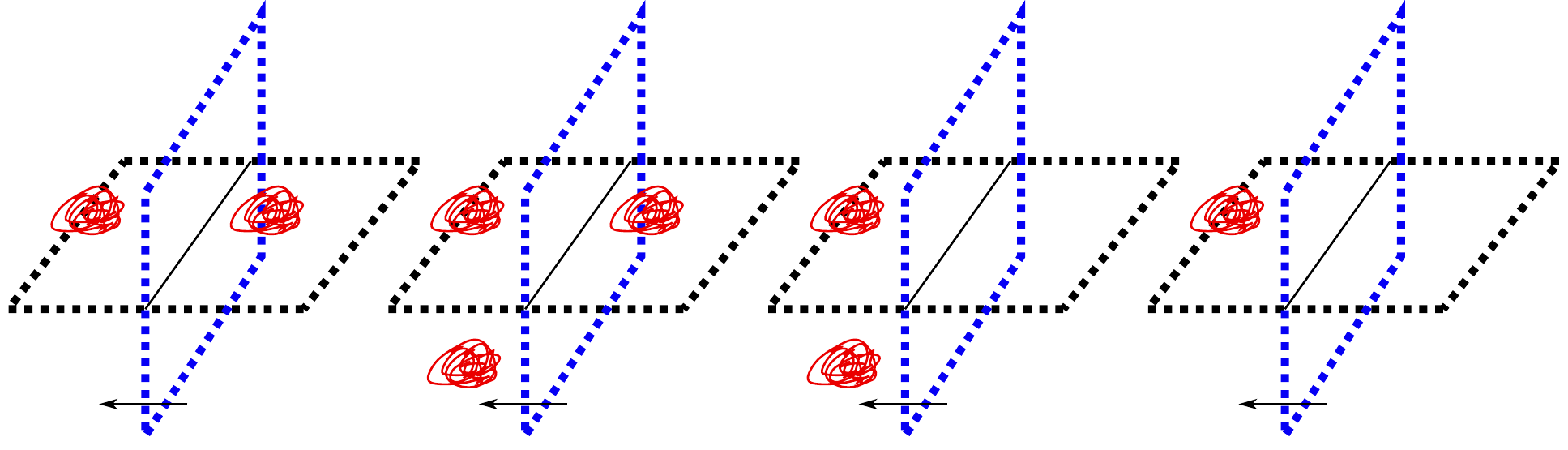}
		\caption{Different possibilities for a neighborhood of a point $p \in S $. Note the two obvious possibilities of having all quadrants or none are not drawn. The horizontal plane is the fully marked surface $S$ with positive normal pointing upward, and the vertical plane is a leaf of the foliation.}
		\label{quarters}
	\end{figure}
	\item $\bar\mA$ is a compact manifold with boundary whose interior is $\mathcal{A}$.
	\item $\partial\bar\mA$ consists of finitely many compact leaves $L_1, \cdots , L_n$ of $\mF_1$ and finitely many subsurfaces $S_1, \cdots, S_m$ of $S$.  $\partial\bar\mA$ is disjoint from points of tangency of $\mF$ with $S$.
	\item The normals to the $L_i$'s and $S_j$'s point into $\mathcal{A}$.
	\item Each $S_i$ is an annulus.  Proof:  $\mathcal{F}|S_i$ is a foliation without singularities having $\partial S_i$ as  leaves.
	\item Each component of $L_i$ is an annulus.  Proof: Double $\bar\mA$ along $S\cap \partial \mathcal{A}$ to obtain $V$, which has a non-singular inward pointing vector field.  It follows that each component of $\partial V$ is a torus and hence each $L_i$ is an annulus or torus.  Since $\mathcal{F}$ is taut and $M$ is atoroidal, it has no torus leaves and hence each $L_i$ is an annulus.
	\item Each component $T$ of $\partial\bar\mA$ is a torus bounding a solid torus $W$.  Proof: By ix) and x), $T$ is a torus.  It contains an essential simple closed curve, e.g. a leaf of $\mathcal{F}|S$, hence by Lemma \ref{torus} $W$ is a solid torus.  Note that either $\mA\subset W$ or $W$ is a component of the complement of $\mA$.
	\item The closure of each component $R$ of $S\cap \text{int}(W)$ is a finite union of properly embedded annuli.  Proof:  Figure \ref{quarters} shows that $R$ is a properly embedded surface whose boundary consists of leaves of $\mF|S$ and hence $R$ is $\pi_1$-injective in $M$.  Since $R$ lies in a solid torus this implies that $R$ is an annulus.
	\item The closure $\bar B$ of each component $B$ of $W\setminus S$ is a solid torus. One such $\bar B$ is bounded by a union of finitely many annuli, each of which lies in $S$ or leaves of $\mF$.  All these normals point into $\bar B$ or all these normal point out of $\bar B$.  Proof:  Suppose all the normals to $\partial W$ point out, which is the case when $W\cap \mA=\emptyset$.  An $R$ as in xii), outermost in $W$, cuts off a solid torus $\bar B$.  If the normal to $R$ points out of $\bar B$ we are done.  Otherwise let $W_1$ be the complementary solid torus.  Its boundary consists of annuli lying in either $S$ or leaves of $\mF$ and all the normals point out.  The result then follows by downward induction on the number of components of $S\cap \text{int}(W)$.
	\item $\bar B$ is a bad solid torus.  Proof: By construction $\text{int}( \bar B)\cap S=\emptyset$. If the normal vectors point out of (respectively into) $\bar B$, then viewed in $M_1$, $\partial \bar B\cap S_1\neq\emptyset$ and $\partial \bar B\cap S_0=\emptyset $ (respectively $\partial\bar B\cap S_0\neq\emptyset$ and $\partial \bar B\cap S_1=\emptyset)$.
\end{enumerate} 

\end{proof}

\section{Combinatorics of train tracks}

\subsection{Cutting surfaces}
\label{cutting-surfaces}
For the purpose of exposition, here is the definition of the complexity function of a surface that will be used later. Let $\mathcal{S}$ be the set of compact connected orientable surfaces. For $S \in \mathcal{S}$, denote the genus and the number of boundary components of $S$ with $g(S)$ and $b(S)$ respectively.  Define the complexity function $c_0(S)$ as the ordered pair $(g(S), b(S))$, with the lexicographic order.
We make the convention that for every surface $S \in \mathcal{S}$, $c_0(S) > c_0(\emptyset)$, where $\emptyset$ is the empty set. If $S$ is a compact orientable surface, then order the components of $S$ that are not disks or annuli as $S_1 , \cdots , S_n$ with
\[ c_0(S_1) \geq c_0(S_2) \geq \cdots \geq c_0(S_n), \]
and define the complexity function $c_1(S)$ as the tuple
\begin{equation} 
 c_1(S)= \big( c_0(S_1), c_0(S_2),\cdots, c_0(S_n), c_0(\emptyset), c_0(\emptyset), \cdots \big),
\label{surface-complexity}
\end{equation}
with the lexicographic order. 

\begin{lem}
Let $S$ be a compact orientable surface. Let $\alpha$ be a homotopically essential simple closed curve in $S$ that is not $\partial$-parallel. If we denote the cut-open surface $S \setminus \setminus \alpha$ by $S'$, then  $c_1(S')<c_1(S)$. \qed

\label{cut-surface}
\end{lem}

%
%
%
%

\begin{lem}
Let $S$ be a compact, orientable surface and $F$ be a compact subsurface of $S$ such that no component of $S-F$ is a disk. Then $c_1(F) \leq c_1(S)$. \qed
\label{subsurface}
\end{lem}

\subsection{Train tracks}
Let $F$ be a compact surface.  A \emph{train track} $\tau$ on $F$ is a finite collection of $1$-dimensional CW-complexes and circles disjointly embedded in $F$ such that 
\begin{enumerate}
\item Every vertex is trivalent.
\item At every vertex, there is a well-defined tangent line.
\item Every \emph{complementary region} has nonpositive \emph{index}.
\end{enumerate}
Figure \ref{standard-train-track}, top-left, shows the local picture around a vertex of $\tau$. The region around the vertex with angle zero is called a \emph{cusp}. A \emph{complementary region} is a connected component $R$ of $F - \tau$, and the \emph{index} of $R$ is defined as 
\[\text{Ind}(R)= \chi(R) - \frac{1}{2}(\text{number of cusps of }R). \]
Condition (3) rules out disks and monogons (i.e. a disk with one cusp), but we allow bigons (i.e. a disk with two cusps) and annuli. Condition (1) is called \emph{genericity} in some texts. A \emph{transverse orientation} on a train track is a choice of transverse orientation on each edge such that they are compatible at each vertex. A train track $\tau$ is called \emph{transversely oriented} if $\tau$ comes equipped with a transverse orientation, and \emph{transversely orientable} if there exists a choice of transverse orientation for $\tau$. A curve $\gamma$ is \emph{transverse} to $\tau$ if $\gamma$ intersects $\tau$ transversely and not at the vertices of $\tau$. 

A \emph{branched neighborhood} $N_b(\tau)$ of a train track $\tau$ is a neighborhood modelled locally as in Figure \ref{standard-train-track} bottom-left for the vertices of $\tau$, and it comes with a projection map $\pi \colon N_b(\tau) \rightarrow \tau$ such that for each $p \in \tau$, $\pi^{-1}(p)$ is an interval called a \emph{tie}. A \emph{singular tie} is $\pi^{-1}(v)$ for a vertex $v$ of $\tau$. Beware that some texts use a different definition for a branched neighborhood, in which the boundary of a branched neighborhood is composed of horizontal and vertical parts. 

\begin{figure}
\labellist
\pinlabel $\tau$ at 60 110
\pinlabel $N_b(\tau)$ at 60 10
\pinlabel $tie$ at 160 60
\endlabellist
\includegraphics[width = 2.5 in]{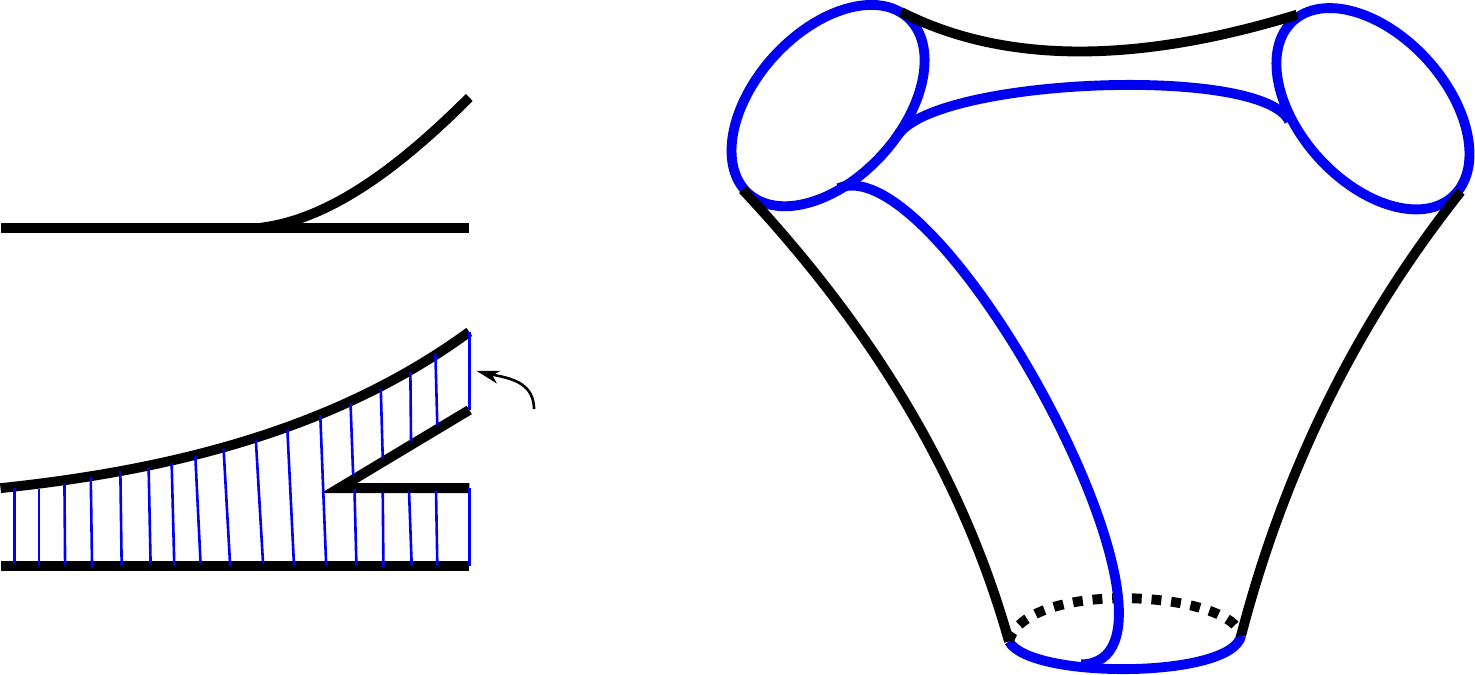}
\caption{Left: the local model for a branched neighborhood of a train track around a vertex, Right: a basic block}
\label{standard-train-track}
\end{figure}

A singular foliation $\mathcal{F}$ is \emph{carried by} a train track $\tau$ if there is a branched neighborhood $N_b(\tau)$ of $\tau$ such that the singularities of $\mathcal{F}$ correspond to the cusps of the branched neighborhood $N_b(\tau)$, the support of $\mathcal{F}$ is equal to $N_b(\tau)$, and $\mathcal{F}$ is transverse to the ties. By cutting the branched neighborhood along singular ties, one obtains the following combinatorial description. Let $e_1 , \cdots , e_n$ be the edges or simple closed curve components of $\tau$. For each edge $e_j$ consider the rectangle $e_j \times I$ with the product foliation, and define its \emph{vertical boundary} as $(\partial e_j) \times [0,1]$. Similarly for each simple closed curve component $e_j$ of $\tau$, consider $e_j \times I$ with the appropriate suspension foliation induced by $\mathcal{F}$. The branched neighborhood $N_b(\tau)$ is obtained from the union $e_j \times I_j$ for $1 \leq j \leq n$, by identifying them suitably along their vertical boundaries. A \emph{rectangle} is the image of an immersion $f \colon [0,1]\times [0,1] \rightarrow N_b(\tau)$ such that 
\begin{enumerate}
\item $f|(0,1)\times (0,1)$ is an embedding.
\item For each $t \in [0,1]$, $f([0,1] \times t)$ is contained in a finite union of leaves and singularities of $\mathcal{F}$; if $t \in (0,1)$ then $f( [0,1] \times t)$ is contained in a single leaf.
\item $f( t \times [0,1])$ is transverse to $\mathcal{F}$ for every $t \in [0,1]$.
\end{enumerate}
In particular for any edge $e_j$ of $\tau$, $e_j \times I$ is a rectangle.

We say $(\mathcal{F}, \tau)$ \emph{splits} to $(\mathcal{F}', \tau')$ if $(\mathcal{F}', \tau')$ is obtained from $(\mathcal{F}, \tau)$ by a finite sequence of moves as shown in Figure \ref{definition-splitting}.

\begin{figure}
\centering
\includegraphics[width =4 in]{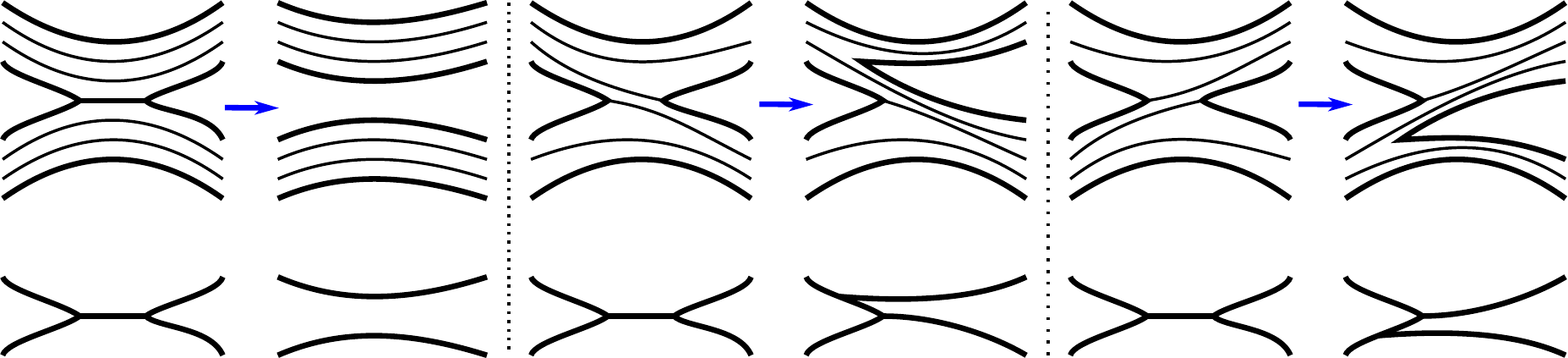}
\caption{Different splitting of a train track (bottom) and the carried foliation (top)}
\label{definition-splitting}
\end{figure}


A complementary region $R$ is called \emph{embedded} if the natural map $\overline{R} \rightarrow F$ is an embedding, where $\overline{R}$ is the closure of $R$. A \emph{smooth annulus} complementary region is an annulus complementary region with no cusps.


A \emph{closed curve} (respectively an \emph{arc}) $\alpha \subset \tau$ is the image of an immersion $S^1 \rightarrow \tau$ (respectively $[0,1] \rightarrow \tau$). We say that $\alpha $ is \emph{smooth} if the induced tangent line on $\alpha$ from $\tau$ is continuous. In other words the image of $\alpha$ has no cusps (respectively has no cusps when restricted to $(0,1)$). When the surface $F$ is orientable and $\tau$ is transversely orientable, a smooth arc whose endpoints coincide is a smooth closed curve as well. This is because a cusp at the endpoint of a smooth arc would be inconsistent with the transverse orientation of $\tau$. Note that a smooth simple closed curve $\alpha \subset \tau$ is homotopically essential, since the complementary regions to $\tau$ have nonpositive index. 

If $\tau$ is a transversely oriented train track, then an oriented closed curve $\gamma$ intersects $\tau$ \textbf{coherently} if $\gamma \cap \tau \neq \emptyset$, $\gamma$ is transverse to $\tau$, and at every point $p \in \gamma \cap \tau$ the orientation of $\gamma$ is consistent with the transverse orientation of $\tau$. An unoriented closed curve $\gamma$ intersects $\tau$ coherently if there is a choice of orientation for $\gamma$ such that all intersections of $\gamma$ with $\tau$ are coherent.

\begin{definition}
Let $F$ be a compact orientable surface, and $\tau$ be a train track on $F$. Define the complexity function  $c(F, \tau)$ as the triple
\[ c(F, \tau) = \big(c_1(F), c_2(\tau) , c_3(\tau) \big),   \]
where $c_1(F)$ is defined as in Equation (\ref{surface-complexity}) in Subsection \ref{cutting-surfaces}, $c_2$ is the number of edges of $\tau$, and $c_3$ is the number of circle components of $\tau$. We order the triples $c(F,\tau)$ lexicographically.
\label{train-track-complexity}
\end{definition}

\begin{definition}
A train track obtained by adding an edge connecting the boundary components of an annulus is called a \textbf{standard train tracked annulus} (Figure \ref{standard-train-tracked-annulus}). If we give a transverse orientation to the train tracks in the right side of Figure \ref{standard-train-tracked-annulus}, the picture on the bottom right admits an outward (respectively inward) pointing transverse orientation (with respect to the ambient annulus), while the picture on the top right admits a mixed transverse orientation. 
\end{definition}

\begin{figure}
\centering
\includegraphics[width = 3 in]{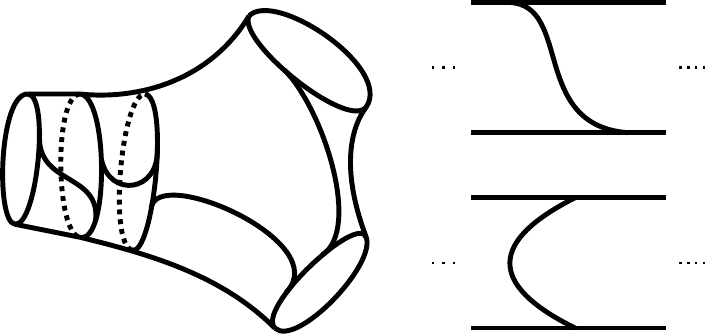}
\caption{Right: the two possible shapes for a standard train tracked annulus, Left: a generalized basic block.}
\label{standard-train-tracked-annulus}
\end{figure}

\begin{definition}
Let $F$ be a compact orientable surface, and $\tau \subset F$ be a train track. The pair $(F, \tau)$ is called a \textbf{basic block} if
\begin{enumerate}
\item $\tau$ is the union of $\partial F$ and finitely many disjoint arcs $\gamma_i$, where each $\gamma_i$ connects two distinct components of $\partial F$, and
\item for each component $b$ of $\partial F$, all the cusps of $\tau$ on $b$ point in the same direction.
\end{enumerate}
A pair $(F, \tau)$ is an \textbf{almost basic block} if it satisfies Condition (1) above. See Figure \ref{standard-train-track} for an example of a basic block.
\label{Definition-basic-block}
\end{definition}

\begin{definition}
A \textbf{generalized basic block} is obtained from a basic block by successive attachment of standard train tracked annuli to its boundary components (Figure \ref{standard-train-tracked-annulus}). 
\end{definition}

\begin{lem}
Let $F_0$ be a compact orientable surface, $\tau_0$ be a transversely orientable train track on $F_0$, and $\mathcal{F}_0$ be a singular foliation carried by $\tau_0$. Let $F$ be a compact subsurface of $F_0$ and $\tau = \tau_0 \cap F$. Assume that 
\begin{enumerate}
\item $(F, \tau)$ is an almost basic block.
\item Any smooth simple closed curve that is a union of edges of $\tau$, is a component of $\partial F$. 
\end{enumerate}
Then $(\mathcal{F}_0, \tau_0)$ splits to $(\mathcal{F}_1, \tau_1)$ such that the splitting is supported on $F$, and $(F, \tau_1 \cap F)$ is a basic block. 
\label{third-condition}
\end{lem}

\begin{proof}
For each component $c$ of $\partial F$, let $f_{\tau}(c)$ be equal to the number of times that the cusps of $\tau$ lying on $c$ change their direction, as we go around $c$. Let $n(\tau)$ be the number of components $c$ of $F$ such that $f_\tau(c)=0$. It is easy to show that if $\tau$ is not a basic block but satisfies the hypothesis of the lemma, then $(\mathcal{F}_0,\tau_0)$ splits to $(\mathcal{F}_1, \tau_1)$ with the splitting supported on $F$ and such that $\tau_1$ also satisfies the hypothesis of the lemma and $n(\tau_1 \cap F)>n(\tau_0 \cap F)$. This implies the lemma.
\end{proof}

\begin{notation}(Active subsurface)
Given a transversely oriented train track $\tau \subset F$ and a branched neighborhood $N_b(\tau)$ of $\tau$, let $C$ be the union of components of $\overline{F - N_b(\tau)}$ that have no cusps. Define the \textbf{active subsurface} $\mathcal{A}_{N_b(\tau)}$ of $\tau$ as $\overline{F - C}$. 
\label{active subsurface}
\end{notation}

\begin{remark}
The active subsurface has tangential and transverse parts inherited from $N_b(\tau)$. The boundary of the tangential/transverse part of the active subsurface comes with a transverse orientation, except at nodes, induced from $\partial N_b(\tau)$.
\end{remark}

The following lemma is a key combinatorial tool for tangentially extending the foliation $\mathcal{F}_1$ of $M_1 = M \setminus \setminus S$ to one with $S$ a union of compact leaves.

\begin{lem}
\label{reduction}
Let $F$ be a compact orientable surface, $\tau$ be a transversely oriented train track on $F$ with $\partial F \subset \tau$, and $\mathcal{F}$ be a singular foliation carried by $\tau$. Then there is a sequence $(\mathcal{F}_i, \tau_i)$ for $1 \leq i \leq k$ starting at $(\mathcal{F}, \tau)$ and ending at $(\mathcal{F}', \tau')$ such that each term is obtained from the previous one by either splitting or collapsing an embedded bigon, and $\tau'$ satisfies at least one of the following:
\begin{enumerate}
\item There is a homotopically essential, non $\partial$-parallel, simple closed curve $\delta \subset F$ that is disjoint from $\tau'$.
\item There is a homotopically essential, non $\partial$-parallel, simple closed curve $\gamma \subset F$ intersecting $\tau'$ coherently. 
\item For each component $K$ of $\mathcal{A}_{N_b(\tau')}$, $(K,\tau' \cap K)$ is a generalized basic block. 
\end{enumerate}
\end{lem}
\begin{proof}
Consider two cases: 

\textbf{Case 1)} There is a non $\partial$-parallel smooth simple closed curve $\alpha$ that is either a component of $\tau$ or a union of edges of $\tau$.

\textbf{Case 2)} Every smooth simple closed curve that is either a component of $\tau$ or a union of edges of $\tau$ must be $\partial$-parallel.

\underline{First we analyze Case 1}. We may split $(\mathcal{F}, \tau)$ to $(\mathcal{F}', \tau')$ and find a simple closed curve $\alpha'$ isotopic to $\alpha$ that is a union of edges of $\tau'$ or a component of $\tau'$, and is disjoint from $\partial F$. Note that $\alpha'$ is not $\partial$-parallel either, since $\alpha$ and $\alpha'$ are isotopic. Pick a side for $\alpha'$ and call it the plus side. There are two subcases:

\begin{enumerate}[a)]
	\item There is no edge of $\tau'$ emanating from $\alpha'$ on the plus side. 
	\item There is at least one edge of $\tau'$ emanating from $\alpha'$ on the plus side. 
\end{enumerate}

If a) happens, let $\delta$ be a curve obtained by pushing $\alpha'$ slightly to the plus side. Note that $\alpha'$ is homotopically essential, and hence so is $\delta$. Therefore $\delta$ satisfies Condition 1. 

If b) happens, there are two subcases: 
\begin{enumerate}[i)]
	\item All edges of $\tau'$ emanating from $\alpha'$ on the plus side spiral in the same direction.
	\item Not all the edges of $\tau'$ emanating from $\alpha'$ on the plus side spiral in the same direction.
\end{enumerate}

If i) happens, let $\gamma$ be a curve obtained by pushing $\alpha'$ slightly to the plus side. Then $(\mathcal{F}', \tau')$ and $ \gamma$ satisfy Condition 2.

If ii) happens, then there is a segment on $\alpha'$ where two adjacent edges on the plus side spiral in opposite directions and point towards each other.  By splitting along one of them the resulting train track still embeddedly carries $\alpha'$ but has one or two less edges on the plus side. Eventually either that side has no edges or they all spiral in the same direction. Therefore we are back in the previous cases. This finishes the proof for Case 1.\\

\underline{Now we consider Case 2}. Collapse embedded bigons one by one until no embedded bigon is left. Let $\mathcal{A(\tau)}$ be the set of smooth simple closed curves $\alpha$ such that either $\alpha$ is a union of edges of $\tau$, or $\alpha$ is a simple closed curve component of $\tau$. In particular $\partial F \subset \mathcal{A(\tau)}$.We claim that any two elements of $\mathcal{A}(\tau)$ are disjoint. To see this, recall that any element of $\mathcal{A(\tau)}$ is parallel to a component of $\partial F$. Therefore, if two elements $\alpha,\beta \in \mathcal{A}(\tau)$ intersect, there must have been an embedded bigon complementary region.  \\

Let $v$ be a vertex of $\tau$. Consider the singular point $p$ of the foliation $\mathcal{F}$ corresponding to $v$, and let $r$ be the singular leaf emanating from $p$ in $\mathcal{F}$. Let $P$ be the projection of $r$ onto $\tau$ under the projection map $\pi \colon N_b(\tau) \rightarrow \tau$. If $P$ is a finite ray ending at a singularity, split $(\mathcal{F}, \tau)$ to reduce the number of vertices. After splitting along all such finite $P$, we obtain $(\mathcal{F}', \tau')$. We show that after splitting, all the vertices of the train track lie on $\mathcal{A}$. Let $w$ be a vertex of $\tau'$ not lying on $\mathcal{A(\tau')}$, if such a vertex exists. Then the ray $P$ starting at $w$ is an infinite smooth arc in $\tau'$ starting at $w$. If $P \cap \mathcal{A}(\tau') \neq \emptyset$, then we can split $(\mathcal{F}', \tau')$ to move $w$ to $\mathcal{A}(\tau')$. If $P \cap \mathcal{A}(\tau') = \emptyset$, then some vertex in $P$ is repeated, implying that there is a smooth simple closed curve in $\tau'$ but not in $\mathcal{A}(\tau')$, which is not possible. Call the new pair $(\mathcal{F}'',\tau'')$. Every vertex of $\tau''$ lies on $\mathcal{A}(\tau'')$.

Recall that every curve in $\mathcal{A}(\tau'')$ is $\partial$-parallel in $F$, and any two elements of $\mathcal{A}(\tau'')$ are disjoint. For any component $c$ of $\partial F$, define $A_c$ as the maximal annulus neighborhood of $c$ in $F$ with $\partial A_c \subset \mathcal{A}$. Let $F_1 = \overline{F-\cup_c A_c}$ and $\tau_1 = \tau'' \cap F_1$. After collapsing embedded bigons, $\tau_1 \cap A_c$ is a union of standard train tracked annuli and smooth annuli attached together. 

The train track $\tau_1$ is the union of $\partial F_1$ and arcs $\gamma_i$ that go between (not necessarily distinct) components of $\partial F_1$. Note that no $\gamma_i$ can connect a component $b$ of $\partial F_1$ to itself; otherwise the ends of $\gamma_i$ spirals around $b$ in different directions and hence one may construct a smooth simple closed curve in $\tau_1$ that is not a component of $\partial F_1$. By Lemma \ref{third-condition}, $(\mathcal{F}'', \tau'')$ splits to $(\mathcal{F}''', \tau''')$ with splitting supported on $F_1$ such that $(F_1, \tau''' \cap F_1)$ is a basic block. Hence Condition 3 is satisfied, and the proof of Case 2 is complete. 
\end{proof}

\section{Construction of the new foliation}

In this section, starting with a fully marked surface $S$ for the foliation $\mathcal{F}$ in the closed 3-manifold $M$, and a complete system of coherent  transversals $\gamma$, we show how to modify the foliation near $S$ to obtain $\mathcal{G}$ so that $S$ is a union of leaves and the plane field of $\mathcal{F}$ is homotopic to that of $\mathcal{G}$. 

We first show that there exists a vector field on $M$ that is transverse to both $S$ and $\mathcal{F}$.  This enables us to readily keep track of the homotopy class, as $\mathcal{G}$ will retain this property.  Next we cut $M$ along $S$ to obtain the compact manifold $M_1$ and the foliation $\mathcal{F}_1=\mathcal{F}|M_1$, such that $\mathcal{F}_1|\partial M_1$ is a singular foliation with saddle singularities. We push $S \subset M$ slightly in both directions to obtain a manifold $N_2 \subset M_1$ with foliation $\mathcal{F}_2$. Next we extend $\mathcal{F}_2$ to $\mathcal{F}_3$ and $N_2$ to $N_3 \subset M_1$ so that the foliation near $\partial N_2$ looks like a manifold with corners, see Figure \ref{corners}. After more extension to say $\mathcal{F}_4$ and $N_4 \subset M_1$, we see that $N_4$ has been foliated with $\partial N_4$ as leaves, except for finitely many transverse vertical annuli. See Figure \ref{spiral}. By using partial $I$-bundle replacements--generalizing some of the ideas in \cite{MR1162560}--we can remove the transverse vertical annuli one by one, to obtain $\mathcal{F}_5$ on $N_5 = M_1$ with $\partial M_1$ as leaves. Finally, $\mathcal{G}$ is the foliation induced from $\mathcal{F}_5$ by regluing $\partial M_1$ to obtain $M$.  

The operations can be done so that $\gamma$ remains a complete system of coherent transversals for the constructed foliation, implying that the new constructed foliation is taut. 

\subsection{A coherent transverse vector field}

\begin{definition}
Let $M$ be a compact orientable 3-manifold, $\mathcal{F}$ be a taut foliation on $M$, and $S$ be a fully marked surface. A vector field $\mathcal{L}$ defined on $M$ is \textbf{coherently transverse} to $\mathcal{F}$ (respectively $S$) if $\mathcal{L}$ is transverse to $\mathcal{F}$ (respectively $S$) and the orientation of $\mathcal{L}$ is compatible with the transverse orientation of $\mathcal{F}$ (respectively positive orientation of $S$).
\end{definition}

\begin{prop}
Let $M$ be a compact orientable 3-manifold, $\mathcal{F}$ be a taut foliation on $M$, and $S$ be a fully marked surface. There is a vector field $\mathcal{L}$ on $M$ that is coherently transverse to both $\mathcal{F}$ and $S$.
\label{vector-field}
\end{prop}

\begin{proof}
Equip $S$ with its positive orientation. Let $g$ be any initial metric on $M$, and $UT_p(M)$ be the unit tangent space to $M$ at $p \in M$. Denote by $n_1(s)$ (respectively $n_2(s)$) the oriented unit normal vector to $\mathcal{F}$ (respectively $S$) at $s \in S$, and by $H_1(s)$ (respectively $H_2(s)$) the open disk in $UT_s(M)$ corresponding to unit length vectors that make an acute angle with $n_1(s)$ (respectively $n_2(s)$). Since $S$ is a fully marked surface, we have $H_1(s) \cap H_2(s) \neq \emptyset$ for each $s \in S$. The proof consists of two steps. 

\emph{Construction of $\mathcal{L}_1$ in a tubular neighborhood of $S$}: $H_1(s) \cap H_2(s)$ is a topological disk, and the collection of all $H_1(s) \cap H_2(s)$ for $s \in S$ forms a disk bundle over $S$. Pick a section of this disk bundle and set $\mathcal{L}_1|S$ to be equal to this section. Let $S \times [-2 , 2 ]$ be a small tubular neighborhood of $S$. Extend $\mathcal{L}_1$ to $S \times [-1,1]$ by parallel transport. Therefore, $\mathcal{L}_1$ is coherently transverse to both $S$ and $\mathcal{F}$.

\emph{Extending $\mathcal{L}_1$ to $M$}: Let $\mathcal{L}_2$ be any vector field that is defined on $M$, and is coherently transverse to $\mathcal{F}$. The tangent bundle $TM$ of $M$ is trivial (\cite{stiefel1935richtungsfelder} or see \cite[Page 148]{milnor2016characteristic}). Fix a trivialization of $TM$ to identify it with $M \times \mathbb{R}^3$. Define $\mathcal{L}$ to coincide with $\mathcal{L}_1$ on $S \times [-1,1]$, and with $\mathcal{L}_2$ on $M \setminus (S \times [-2, 2])$. Moreover, using the identification $TM \cong M \times \mathbb{R}^3$ define
\[ \mathcal{L}(s,t) = (|t|-1)\mathcal{L}_2 + (2-|t|) \mathcal{L}_1 \hspace{3mm} \text{for  } s \in S, \hspace{2mm} |t| \in [1,2]. \]
Note that at $|t| =1$ (respectively $|t|= 2$), $\mathcal{L}$ coincides with $\mathcal{L}_1$ (respectively $\mathcal{L}_2$). Moreover, as both $\mathcal{L}_1$ and $\mathcal{L}_2$ are coherently transverse to $\mathcal{F}$, then so is any convex combination of them. In particular $\mathcal{L}$ is nonzero at every point. This completes the proof.
\end{proof}

\begin{prop}
Let $M$ be a compact orientable $3$-manifold and $\mathcal{L}$ be a vector field on $M$. Assume that $\mathcal{F}$ and $\mathcal{G}$ are transversely oriented codimension-one foliations on $M$. If both $\mathcal{F}$ and $\mathcal{G}$ are coherently transverse to $\mathcal{L}$, then the oriented tangent plane fields to $\mathcal{F}$ and $\mathcal{G}$ are homotopic.
\label{homotopic}
\end{prop}

\begin{proof}
The tangent bundle of any compact orientable 3-manifold is trivial. Fix such a trivialization $TM \cong M \times \mathbb{R}^3$. Choose a Riemannian metric $g$ on $M$, and let $B \subset M \times \mathbb{R}^3$ be the bundle over $M$ consisting of vectors that make an acute angle with the vector field $\mathcal{L}$. The fiber $F$ of $B$ can be identified with the set of vectors in $\mathbb{R}^3$ with positive $z$ coordinates, and the bundle $B$ can be identified with $M \times F$. Denote the oriented normal vector fields to $\mathcal{F}$ and $\mathcal{G}$ by $n_1$ and $n_2$ respectively. Both $n_1$ and $n_2$ are sections of the bundle $B$, and hence can be homotoped to each other by a straight line homotopy (using the linear structure of $F$ induced from $g$).  The homotopy between $n_1$ and $n_2$ defines a homotopy between the oriented plane fields of $\mathcal{F}$ and $\mathcal{G}$.
\end{proof}

\subsection{Operations for changing the foliation}\ \\

In this section, we will create a $5$-tuple of data to keep track of: a manifold $N$ with boundary sitting inside a compact collar neighborhood $M_1$ of $N$, a foliation $\mathcal{F}$ of $N$, a vector field $\mathcal{L}$ transverse to $\mathcal{F}$, a finite union $\gamma$ of $1$-manifolds that are transverse to $\mathcal{F}$, as well as a train track in $\partial N$. The vector field $\mathcal{L}$ keeps track of the homotopy class of the plane field of the foliation $\mathcal{F}$, the $1$-manifold $\gamma$ keeps track of tautness, and the train track $\tau$ records how the leaves of $\mathcal{F}$ intersect $\partial N$. 

We will construct operations that extend the foliation and simplify the $5$-tuple to eventually get a foliation of a submanifold $N \subset N' \subset M_1$ that has $\partial N'$ as leaves (i.e. $\tau'=\emptyset$) and $M_1 - N'$ is a tubular neighborhood of $\partial M_1$. During the operations, $\mathcal{L}$ and $\gamma$ always remain fixed. Then we can easily extend the foliation to $M_1$ by filling in the region $M_1 - N'$ with a product foliation, and finally we glue the boundary components $S_0$ and $S_1$ of $M_1$ together to obtain a foliation of $M$ that has $S$ as a leaf. Before reading the technical details that follows, we recommend the reader to see the proof of Theorem \ref{fullymarked} at the end of this section, to get an idea of how things fit together. \\

Throughout the section, let $(M_1, N, \tau, \mathcal{L}, \gamma)$ be the following given data: 
\begin{enumerate}
\item $N$ is a compact, orientable $3$-manifold. 
\item $\tau \subset \partial N$ is a transversely oriented train track with a branched neighborhood $N_b(\tau)$.
\item $N_b(\tau)$ induces a \emph{cornered structure} on $\partial N$ modelled as in Figure \ref{corners}. The cornered manifold $N$ comes equipped with a \emph{transverse orientation} on $\partial N$, coherent with $N_b(\tau)$. By definition, this means that the transverse orientation is transverse to $\partial_\tau N$ and tangent to the ties of $N_b(\tau)$.
\item $N$ is a cornered submanifold of the interior of the smooth manifold $M_1$, where $M_1 - N$ is contained in a regular neighborhood $N(\partial M_1) = \partial M_1 \times [0,1]$ of $\partial M_1$.
\item $\mathcal{L}$ is a vector field on $M_1$ whose flow lines induce the vertical fibration of $N(\partial M_1) = \partial M_1 \times [0,1]$. The vector field $\mathcal{L}$ is coherent with the transverse orientation of $\partial N$ as well as with the cornered structure of $\partial N$. Each vertical fiber of $\partial M_1 \times [0,1]$ intersects $N$ in a proper connected interval starting on $\partial M_1 \times 0$.
\item $\gamma$ is a set of disjoint oriented simple closed curves or properly embedded arcs in $M_1$, where the orientations of arcs at their endpoints are coherent with the transverse orientation of $\partial N$, and the restriction of $\gamma$ to $N(\partial M_1)= \partial M_1 \times [0,1]$ is a union of vertical fibers. 
\end{enumerate}

\begin{remark}
Note in (2), the cornered structure on $\partial N$ induced by $N_b(\tau)$ is meaningful even before considering a singular foliation on $N$. 
\end{remark}

\begin{definition}
A transversely oriented possibly singular foliation $\mathcal{F}$ is \textbf{compatible} with $(M_1, N,\tau, \mathcal{L}, \gamma)$ if 
\begin{enumerate}
\item (Boundary condition) $\mathcal{F}$ is a possibly singular foliation of $N$ whose restriction to the interior of $N$ has no singularities. The foliation $\mathcal{F}$ is transverse to $\partial N$ along the interior of $N_b(\tau)$, and tangential to $\partial N$ along $\partial N - N_b(\tau)$. The cusps of $N_b(\tau)$ correspond to the nodes of $\mathcal{F}$, and the edges and simple closed curve components of $\partial N_b(\tau)$ correspond to convex and concave corners for  $\mathcal{F}$. The singular foliation $\mathcal{F}|N_b(\tau)$ is transverse to the ties of the branched neighborhood $N_b(\tau)$, and their transverse orientations are coherent.
\item (Homotopy condition) $\mathcal{F}$ is transverse to $\mathcal{L}$, and its transverse orientation is coherent with the orientation of $\mathcal{L}$. 
\item (Tautness condition) $\mathcal{F}$ is transverse to $\gamma$, and its transverse orientation is coherent with the orientation of $\gamma$. Moreover every leaf of $\mathcal{F}$ intersects $\gamma$. 
\end{enumerate}
\label{compatible-foliation}
\end{definition}

\begin{definition}
Assume that $\mathcal{F}$ is compatible with $(M_1, N,\tau, \mathcal{L},\gamma)$. We say that $(\mathcal{F'}, N' , \tau')$ is an \textbf{$(\mathcal{L}, \gamma)$-extension} of $(\mathcal{F}, N , \tau)$ if $N \subset N' \subset M_1$ and $\mathcal{F}'$ is compatible with $(M_1, N',\tau', \mathcal{L}, \gamma)$ .
\end{definition} 
In what follows $\mathcal{F}'$ will be obtained by partial $I$-bundle replacement and an extension of $\mathcal{F}$ to $N' - N$.

Throughout the section, we will do a sequence of modifications to $(\mathcal{L}, \gamma)$-extend the foliation such that at the end the boundary of the manifold $N$ becomes a union of leaves.

\begin{operation}(Splitting the train track)
Let $\mathcal{F}$ be a singular foliation compatible with $(M_1, N,\tau, \mathcal{L}, \gamma)$. Let $\mathcal{G}$ be the restriction of $\mathcal{F}$ to $\partial N$, and assume that $(\mathcal{G}, \tau)$ splits to $(\mathcal{G}', \tau')$. There is an $(\mathcal{L}, \gamma)$-extension $(\mathcal{F}', N', \tau')$ of $(\mathcal{F}, N, \tau)$ such that the restriction of $\mathcal{F}'$ to $\partial N'$ is homeomorphic to $\mathcal{G}'$. \qed
\label{extension-splitting}
\end{operation}

\begin{lem}
Let $F$ be a compact orientable surface, $\tau$ be a train track on $F$ with a branched neighborhood $N_b(\tau)$, and $\mathcal{F}$ be a singular foliation carried by $\tau$. If $(\mathcal{F}, \tau)$ splits to $(\mathcal{F}', \tau')$ then $c(\mathcal{A}_{N_b(\tau')}, \tau') \leq c(\mathcal{A}_{N_b(\tau)} , \tau) $.  \qed
\label{splitting-complexity}
\end{lem}


Define a \textbf{ditch} in $\partial M$ as $A \cup L \cup B$ such that $L \subset \partial_\tau M$ is an annulus with $\partial L = \{ \alpha , \beta \}$ such that 
\begin{enumerate}
\item $A, B \subset \partial_\pitchfork M$ are annuli foliated as suspensions with $L \cap A = \alpha $ and $ L \cap B = \beta$. 
\item The points of $\alpha \cup \beta$ (respectively $\partial (A\cup B) - \{\alpha, \beta \}$) are concave (respectively convex) corners.
\end{enumerate}

\begin{operation} (Cutting the active subsurface)
	Let $\mathcal{F}$ be a singular foliation compatible with $(M_1, N,\tau, \mathcal{L}, \gamma)$ . Let $\delta$ be a homotopically essential, non $\partial$-parallel, simple closed curve in $\mathcal{A}_{N_b(\tau)}$ that is disjoint from $\tau$. Denote the component of $\partial_\tau M$ containing $\delta$ by $T$. There is an $(\mathcal{L}, \gamma)$-extension $(\mathcal{F}', N', \tau')$ of $(\mathcal{F}, N, \tau)$ such that, up to a homeomorphism identifying $\partial N$ with $\partial N'$, $\tau'$ is obtained as follows:
	\begin{enumerate}
		\item If $\delta \subset T$ is separating and for some component $T_1$ of $T \setminus \delta$ the restriction of the transverse orientation of $\partial T$ to $T_1$ always points in, then $\tau' = \tau \cup \delta$ with the transverse orientation of $\delta$ pointing out of $T_1$.
		
		\item Otherwise, $\tau'$ is obtained from $\tau$ by adding $\partial N(\delta)$ with transverse orientation pointing out of $N(\delta)$, where $N(\delta)$ is a tubular neighborhood of $\delta$ in $T$.
	\end{enumerate}
	In both cases, $c_1(\mathcal{A}_{N_b(\tau')})< c_1(\mathcal{A}_{N_b(\tau)})$.
	\label{extension-cutting-the-surface}
\end{operation}

\begin{remark}
	Note $T_1$ can not have any cusp. If $\delta$ is non-separating in $T$, then $T \setminus \delta$ is connected and has at least one cusp since $T \subset \mathcal{A}_{N_b(\tau)}$ had at least one cusp; so we are in the Case (2) above. If $\delta$ is separating in $T$, then at least one of the two components of $T \setminus \delta$ has a cusp; as a result there is at most one choice for $T_1$.
\end{remark}

\begin{proof}
	Let $L$ be the leaf of $\mathcal{F}$ containing $T$. In Case (1), do a partial $I$-bundle replacement along $L - \text{int}(T_1)$. Therefore, $\mathcal{A}_{N_b(\tau')} = \mathcal{A}_{N_b(\tau)}- \text{int}(T_1)$. If $T_1$ has negative Euler characteristic, we have $c_1(\mathcal{A}_{N_b(\tau')})<c_1(\mathcal{A}_{N_b(\tau)})$.
	
	In Case (2), do a partial $I$-bundle replacement along $L - N^\circ(\delta)$ where $N^\circ(\delta)$ is the interior of $N(\delta)$. This has the effect of creating a ditch around $\delta$. In this case, $\mathcal{A}_{N_b(\tau')}$ is a subsurface of $\mathcal{A}_{N_b(\tau)} - N^\circ(\delta)$, and Lemmas \ref{subsurface} and \ref{cut-surface} imply that
	\[ c_1(\mathcal{A}_{N_b(\tau')})\leq c_1(\mathcal{A}_{N_b(\tau)} - N^\circ(\delta)) < c_1(\mathcal{A}_{N_b(\tau)}).\]

	By the hypothesis, every added leaf is attached to one of the former leaves and hence they share a transversal. This was the motivation for considering Case (1) separately. 
\end{proof}

\begin{operation}(Embedded bigons)
Let $\mathcal{F}$ be a singular foliation compatible with $(M_1, N, \tau, \mathcal{L}, \gamma)$. Assume that some complementary region $B$ of $\tau$ is an embedded bigon. There is an $(\mathcal{L}, \gamma)$-extension $(\mathcal{F}', N', \tau')$ of $(\mathcal{F}, N, \tau)$ such that, up to a homeomorphism identifying $\partial N'$ with $\partial N$, $\tau'$ is obtained from $\tau$ by collapsing the bigon $B$. In particular $c_1(\mathcal{A}_{N_b(\tau')}) \leq c_1(\mathcal{A}_{N_b(\tau)})$ and $c_2(\tau')< c_2(\tau)$. \qed
\label{extension-embedded-bigon}
\end{operation}

\begin{operation} (Spiralling)
Let $\mathcal{F}$ be a singular foliation compatible with $(M_1, N,\tau, \mathcal{L}, \gamma)$ and $\delta \subset \mathcal{A}_{N_b(\tau)}$ be a simple closed curve that intersects $\tau$ coherently. There is an $(\mathcal{L}, \gamma)$-extension $(\mathcal{F}', N', \tau')$ of $(\mathcal{F}, N, \tau)$ where $\tau'$ is defined, up to a homeomorphism identifying $\partial N'$ with $\partial N$, as follows. Denote a small neighborhood of $\delta$ in $\partial N$ by $N(\delta)$. Let $\tau'$ be obtained from $\tau$ by deleting $ \tau \cap N(\delta)$ and adding $\partial N(\delta)$ with inward pointing transverse orientation (Figure \ref{spiralling-train-track}). In particular if  $\delta \subset \mathcal{A}_{N_b(\tau)}$ is non $\partial$-parallel then $c_1(\mathcal{A}_{N_b(\tau')})<c_1(\mathcal{A}_{N_b(\tau)})$. 
\label{extension-spiralling}
\end{operation}

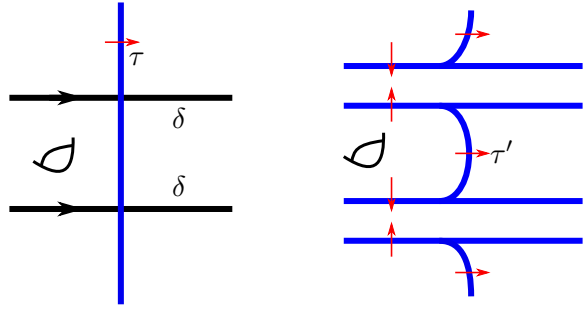
\begin{figure}
\labellist
\pinlabel $\tau$ at 75 145
\pinlabel $\tau'$ at 290 90
\pinlabel $\delta$ at 100 110
\pinlabel $\delta$ at 100 70
\endlabellist
\centering
\includegraphics[width= 3 in]{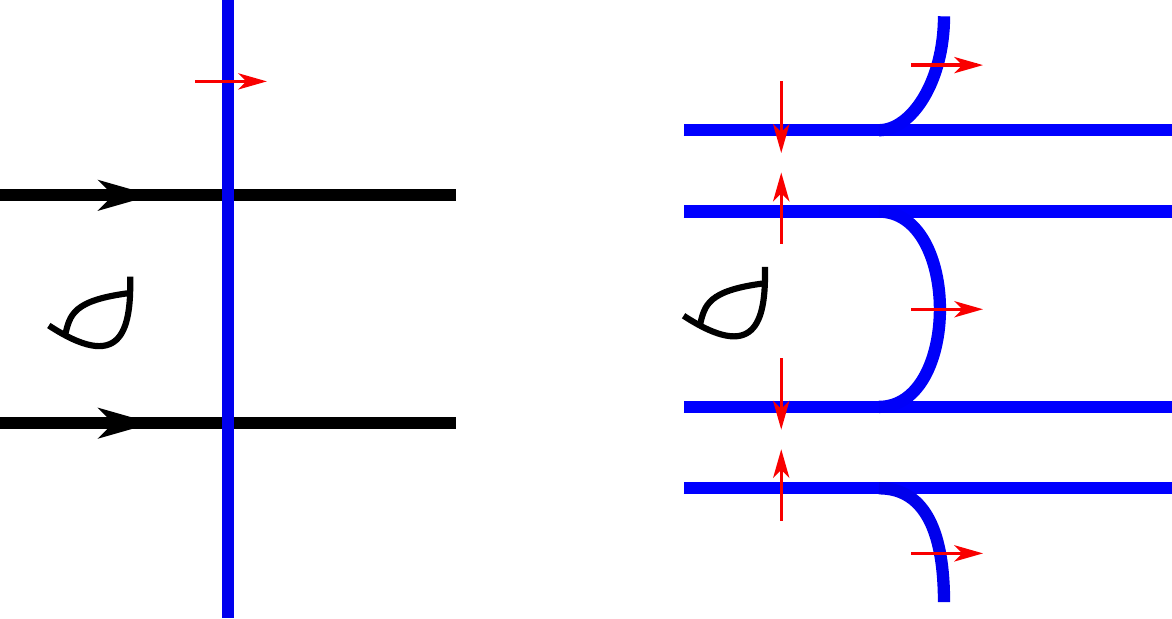}
\caption{In the left, the vertical line is part of the transversely oriented train track $\tau$, and the horizontal lines are part of the curve $\delta$ intersecting $\tau$ coherently. In the right, the new train track after spiralling is drawn.}
\label{spiralling-train-track}
\end{figure}

\begin{proof}
This is a relative version of ``turbulization'', due to Reeb. See \cite[Figure 5.3.]{gabai1983} or \cite[Figure 6]{novikov1965topology}. Spiral the leaves of $\mathcal{F}|N(\delta)$ around $N(\delta)$ to obtain the desired $\mathcal{F}'$. 

Since $\delta \subset \mathcal{A}_{N_b(\tau)}$, we have $\mathcal{A}_{N_b(\tau')} \subset \mathcal{A}_{N_b(\tau)} - N^\circ(\delta)$. By Lemma \ref{subsurface}, $c_1(\mathcal{A}_{N_b(\tau')}) \leq c_1(\mathcal{A}_{N_b(\tau)} - N^\circ(\delta))$. Note that $\delta$ is homotopically essential since it has nonzero algebraic intersection number with $\tau$. If $\delta$ is non $\partial $-parallel in $\mathcal{A}_{N_b(\tau)}$, by Lemma \ref{cut-surface} we have $c_1(\mathcal{A}_{N_b(\tau)} - N^\circ(\delta))< c_1(\mathcal{A}_{N_b(\tau)})$, which implies $c_1(\mathcal{A}_{N_b(\tau')})<c_1(\mathcal{A}_{N_b(\tau)})$.
\end{proof}

The following operation is a generalization of a construction due to the first author; see \cite[Page 476, Figures a-c]{gabai1983}.

\begin{operation}(Basic blocks)
	Let $\mathcal{F}$ be a singular foliation compatible with $(M_1, N, \tau, \mathcal{L}, \gamma)$, $F \subset \mathcal{A}_{N_b(\tau)}$ be a compact subsurface, and $(F, \tau \cap F)$ be a basic block with $\chi(F)<0$. There is an $(\mathcal{L}, \gamma)$-extension $(\mathcal{F}', N', \tau')$ of $(\mathcal{F}, N, \tau)$ where, up to a homeomorphism identifying $\partial N'$ with $\partial N$, $\tau'$ coincides with $\tau$ outside of $F$, and the restriction of $\tau'$ to $F$ is supported in a tubular neighborhood of $\partial F$ in $F$ in the following manner: Let $b$ be a component of $\partial F$ with the tubular neighborhood $N(b)$, $\partial N(b) = \{ b , b' \}$ with the transverse orientation of $b'$ pointing out of $N(b)$,
	
	\begin{enumerate}
		\item If the transverse orientation of $b$ points into $F$, then $\tau'_{|N(b)} = b $. 
		\item If the transverse orientation of $b$ points out of $F$ and $b$ has no edges to the outside or inside ($b$ is a simple closed curve component of $\tau$), then $\tau'_{|N(b)}=  b \cup b' $.
		\item If the transverse orientation of $b$ points out of $F$ and $b$ has no edges to the outside but has at least one edge to the inside, then $\tau'_{|N(b)} = b  $ or $b \cup b'$.
		\item If the transverse orientation of $b$ points out of $F$ and $b$ has at least one edge to the outside, then $\tau'_{|N(b)} = \tau_{|N(b)} \cup b' $. 
	\end{enumerate}
	See Figure \ref{basic-block-operation-figure} where items (1)-(4) are shown from left to right.
	In particular $c_1(\mathcal{A}_{N_b(\tau')})<c_1(\mathcal{A}_{N_b(\tau)})$. 
	
	If $(F,\tau \cap F)$ is a standard train tracked annulus with $\partial F = \{ b ,b'\}$, then we can find the $(\mathcal{L}, \gamma)$-extension in the following cases:
	\begin{enumerate}[i)]
		\item If the transverse orientation of $\tau \cap F$ points into $F$, then $\tau'_{|F} = b \cup b' $.
		\item If the transverse orientation of $\tau \cap F$ is mixed, $b$ is the outward pointing component of $\partial F$, and $b$ has no edges to the outside, then $\tau'_{|F}= b'$.
		\item If the transverse orientation of $\tau \cap F$ points outside $F$, and neither $b$ nor $b'$ has edges to the outside, then $\tau'_{|F} \subset b \cup b'$.
	\end{enumerate}
	\label{extension-basic-block}
\end{operation}

\begin{figure}
	\includegraphics[width=3.5 in]{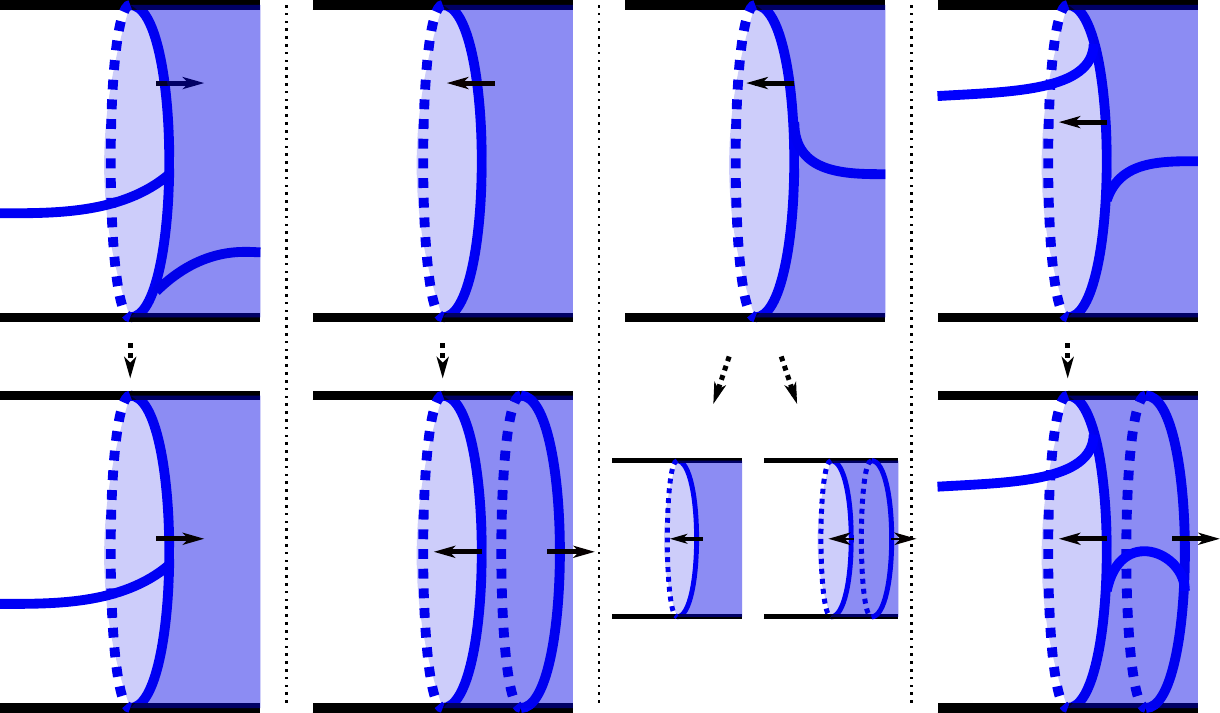}
	\caption{Four possibilities for the Operation \ref{extension-basic-block}, from left to right. The solid color indicates the surface $F$, and the middle curve is $b$. The arrows show the transverse orientation of the train track.}
	\label{basic-block-operation-figure}
\end{figure}

\begin{proof}
	Consider the case of $\chi(F)<0$, as the other case is similar.
	
	\textit{Special Case:} First we consider the case that the transverse orientation of $\tau_1= \tau \cap F$ points into $F$ along $ \partial F$; so we are in Case (1) for all boundary components of $F$. Let $\gamma_i \subset \tau$ be the arcs going between different boundary components of $F$ as in the definition of a basic block. Let $\gamma_i \times [0,1]$ be the rectangle corresponding to $\gamma_i$ in the branched neighborhood $N_b(\tau)$ of $\tau$. Denote the sides $\gamma_i \times 0$ and $ \gamma_i \times 1$ by $\gamma_i^+$ and $\gamma_i^-$ respectively, and assume that the transverse orientation points from $\gamma_i^-$ to $\gamma_i^+$. Denote the tangential part of $F$ by $T_1$. To visualize the construction easier, we think of the rectangles $\gamma_i \times [0,1]$ as perpendicular to both $T_1$ and the rectangle $e \times [0,1]$ for every edge $e \in \partial F$ adjacent to $\gamma_i$ (Figure \ref{extension-basic-block-figure}). 
	
	\begin{figure}
		\labellist
		\pinlabel $\gamma_i^-$ at 50 90
		\pinlabel $\gamma_i^+$ at 110 135
		\endlabellist
		\centering
		\includegraphics[width= 2 in]{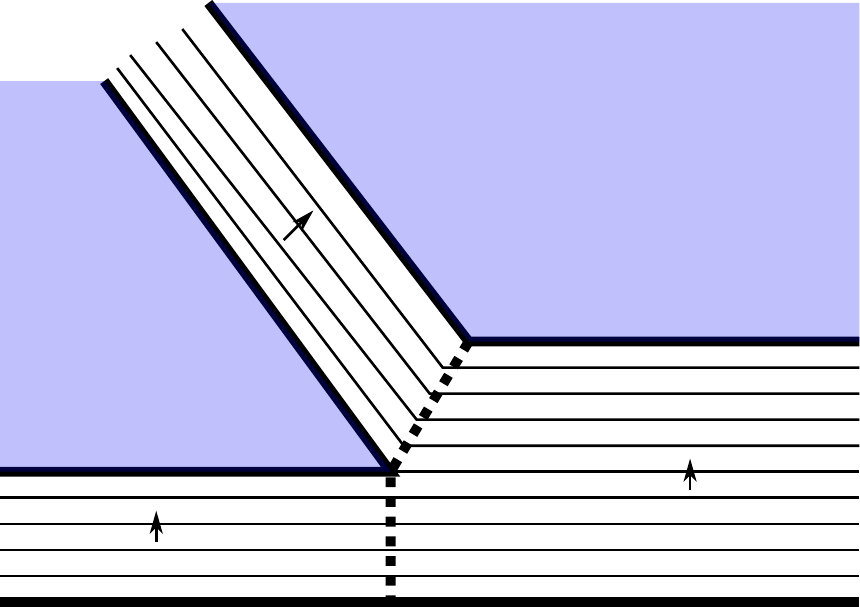}
		\caption{The relative position of the arcs $\gamma_i^+$ and $\gamma_i^-$. The arrows show the transverse orientation of the train track, and the solid color indicates the tangential part $T_1$. The broken dashed line is a singular tie.}
		\label{extension-basic-block-figure}
	\end{figure}
	
	Let $N(\gamma_i^-)$ be a tubular neighborhood of $\gamma_i^-$ in $T_1$ with boundary $\gamma_i^- \cup \gamma_i^*$ and set 
	\[ T_2 = \overline{T_1 - \bigcup_i N(\gamma_i^-)}. \]
	Let $t$ be a component of $\partial T_2$ that contains at least one of $\gamma_j^*$ or $\gamma_j^+$ for some $j$. It follows from the definition of a basic block that $t$ is a union of two types of transversely oriented arcs in an alternate fashion:
	\begin{enumerate}
		\item one of the arcs $\gamma_i^+$ or $\gamma_i^*$, running between two distinct components of $\partial F$, or
		\item an arc corresponding to a subset of $\partial F$.
	\end{enumerate}
	Let $I$ be a closed interval. Attach a copy of $T_2 \times I$, with the product foliation, to $N$ by identifying $T_2 \times 0$ with $T_2 \subset \partial N$. After the attachment, there are \emph{walls} on top of each of the arcs $\gamma_i^*  $ and $\gamma_i^-$. By definition, the wall above $\gamma_i^*$ is $\gamma_i^* \times I \subset T_2 \times I$. The wall above $\gamma_i^-$ is the union $\big(\gamma_i \times [0,1] \big) \cup_{\gamma_i^+} \big( \gamma_i^+ \times I \big)$. Here the first term comes from the rectangle corresponding to $\gamma_i$, the second term comes from the restriction of $T_2 \times I$ to $\gamma_i^+$, and the two terms are glued along their common arc $\gamma_i^+$. Let $J$ be a closed interval with initial and terminal points $i(J)$ and $t(J)$ respectively. Now we connect the walls above $\gamma_i^*$ and $\gamma_i^-$ by attaching $N(\gamma_i^-) \times J$ equipped with the product foliation to $N \cup_{T_2} (T_2 \times I)$ such that 
	\begin{enumerate}
		\item $N(\gamma_i^-) \times i(J)$ is identified with $N(\gamma_i^-) \subset T_1$, and
		\item $\gamma_i^- \times J$ (respectively $\gamma_i ^* \times J$) is identified with the wall above $\gamma_i^-$ (respectively $\gamma_i^*$) .
	\end{enumerate}
	This operation replaces $\tau$ with $\tau' = \tau - \cup_i \gamma_i$. In particular, $c_1(\mathcal{A}_{N_b(\tau')})<c_1(\mathcal{A}_{N_b(\tau)})$.  Note we did not use the consistent spiraling (Condition 2 in the definition of a basic block) in this special case.\\
	
	\textit{General Case:} Let $B$ be the union of components of $\partial F$ whose transverse orientation points out of $F$. Remove a tubular, possibly cornered, neighborhood of $B$ from $F$ to obtain $F'$. See Figure \ref{basic-block-nbhd}, where $F'$ is the part of the surface that lies below the broken dashed line $b^+$. Do the operations as in the \textit{Special Case} for $F'$ by pretending that the transverse orientation points into $F'$, to obtain a foliation $\mathcal{F}_1$. The accumulation of leaves of $\mathcal{F}_1$ creates a `wall' on top of $b^+$, with top boundary component $\hat{b}$. If $b \in B$ corresponds to Item (2) or (4) in the statement, then the new train track in a neighborhood of $b$ has the claimed description. 
	
	Now let $b \in B$ correspond to Item (3) in the statement, so $b$ has no edges to the outside. Let $N(b)$ (respectively $N'(b)$) be the annulus cobounded by $b^*$ and $b^+$ (respectively $b_*$). Spiral the leaves of $\mathcal{F}_1$ intersecting $N(b)$ around $N(b)$ to converge to an annulus $A$ such that $\partial A= \{ \hat{b} \cup b_{\infty} \}$ where $b_\infty$ is a leaf of the induced foliation on $N'(b)$. Such leaf $b_\infty$ exists since we assumed that $b$ has no edges to the outside. The new train track has either one copy of $b$ (when $b_* \neq b_\infty$), or no copy of $b$ (otherwise). If the new train track has no copy of $b$, then use Operation \ref{extension-cutting-the-surface} for the curve $b$. The purpose of using Operation \ref{extension-cutting-the-surface} is to make sure $c_1(\mathcal{A}_{N_b(\tau)})$ reduces when $\chi(F)<0$ (if the complementary region $R$ to the outside of $b$ has a cusp, and if we remove $b$ from $\tau'$ altogether then $R$ is merged with $F'$ to form a complementary region $R'$. But now $R'$ has a cusp and hence is included in the active subsurface of $\tau'$.)  The result is one of the two pictures in Item (3). This completes the construction of $(\mathcal{F}', \tau')$.
	
	
	%
	%
	
	\begin{figure}
		\labellist
		\pinlabel $b^+$ at 170 52
		\pinlabel $b^*$ at 170 65
		\pinlabel $b_*$ at 170 110
		\endlabellist
		\begin{center}
			\includegraphics[width=2 in]{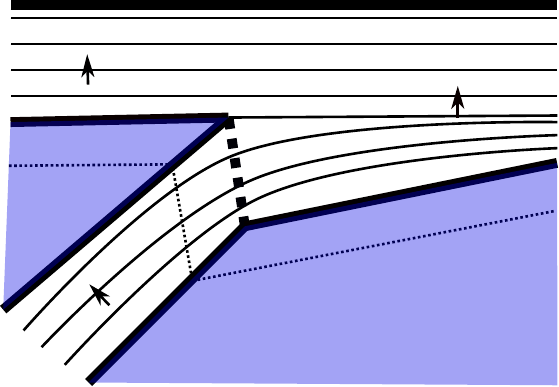}
		\end{center}
		\caption{The arrows indicate the transverse orientation of the train track, and the solid color shows the tangential part $T_1$.}
		\label{basic-block-nbhd}
	\end{figure}

\end{proof}

\begin{remark}
Let $R$ be a three-dimensional Reeb component and $\alpha_i \subset \partial R$ be a set of disjoint curves parallel to the core of $R$. Let $B \subset M_1$ be a bad solid torus whose foliation is homeomorphic to the foliation obtained by shaving a neighborhood of all $\alpha_i$ in $R$. Assume that $B$ intersects $S$ in annuli $A_i$ for $1 \leq i \leq n$ where the induced foliation on each $A_i$ is a two-dimensional Reeb component. If we use Operation \ref{extension-basic-block} to tangentially extend the foliation along $A_i$, we may reproduce the Reeb component $R$. Therefore, the constructed foliation may not be taut. This is why we did an initial preparation to ensure that no bad solid torus exists.
\end{remark}

\begin{operation} (Spinning)
Let $\mathcal{F}$ be a singular foliation compatible with $(M_1, N, \tau, \mathcal{L}, \gamma)$, and $\partial N$ be incompressbile. Assume that $\tau$ is a union of disjoint simple closed curves (e.g. Figure \ref{spiral}, left). There is an $(\mathcal{L}, \gamma)$-extension $(\mathcal{F}', N', \tau')$ of $(\mathcal{F}, N, \tau)$ with $\tau' = \emptyset$.
\label{extension-spinning}
\end{operation}
%
%

\begin{proof} 
		Any simple closed curve in $\tau$ is homotopically nontrivial in $\partial N$, since otherwise some complementary region to $\tau$ has to be a disk. Since $\partial N$ is incompressible, any simple closed curve in $\tau$ is homotopically nontrivial in $N$ as well. Pick a component $\alpha$ of $\tau$, and let $A$ be the maximal regular neighborhood of $\alpha$ such that $A \subset \partial_\pitchfork N$. Denote the holonomy of $A$ by $\mu$, and let $\partial A = a \cup b$, where the points of $a$ (respectively $b$) are concave (respectively convex) corners. See Figure \ref{spiral}. We show how to replace $\tau$ by $\tau - \alpha$. 
		
		\begin{figure}
			\begin{minipage}{5.5 cm}
				\centering
				\includegraphics[width=2 in]{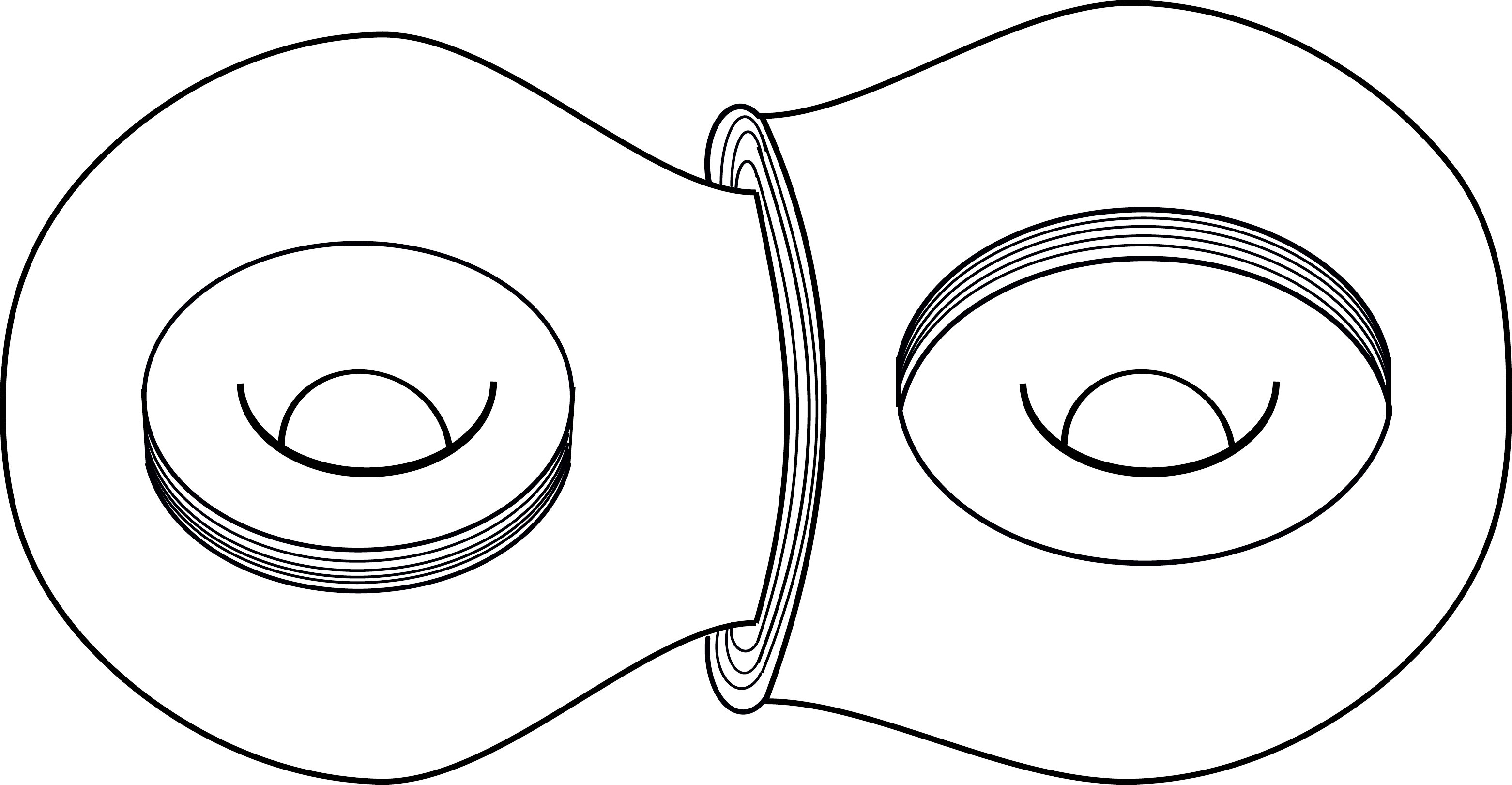}
			\end{minipage}
			\begin{minipage}{5.5 cm}
				\labellist
				\pinlabel $L_3$ at 75 27
				\pinlabel $a$ at 62 27 
				\pinlabel $b$ at 47 27
				\endlabellist
				
				\centering
				\includegraphics[width=2 in]{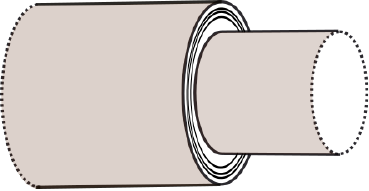}
			\end{minipage}
		
		\caption{}
			\label{spiral}
		\end{figure}
	
	Let $L_1$ be the leaf of the foliation $\mathcal{F}$ containing $a$ and set $L_2:= L_1 \setminus \setminus a$. Let $L_3$ be the connected component of $L_2$ which is on the tangent boundary side of $a$ (Figure \ref{spiral}, right). Note that $L_3$ can be a proper subset of $L_2$ or all of it. Take a small tubular neighborhood $N(a)$ of $a$ in $L_3$ with $\partial N(a)= a \cup c$, and set $L:= L_3 - N^\circ(a)$.
	
	 We do a partial $I$-bundle replacement for $L$ and foliate it to have certain holonomies on different boundary components, in such a way that the new holonomy along $c$ becomes conjugate with the new holonomy of $A$. Consider three different cases, where in each case the $I$-bundle exists by Lemma \ref{transversefoliations}: 
\begin{enumerate}
	\item $L$ is not compact planar. Let the foliated $I$-bundle over $L$ have holonomy $\eta$ on $c$ and identity on all other boundary components, where $\eta$ = $\mu '$ and $\mu '$ is obtained from $\mu$ by trivial Denjoy blow-up along all components of $L \cap A$.
	
	\item $L$ is compact planar and some component $d$ of $\partial L - c$ is disjoint from $A$. Let the foliated $I$-bundle over $L$ have holonomy $\eta$ on $c$ and $\eta ^{-1}$ on $d$ and identity on all other boundary components, where $\eta = \mu'$.
	
	\item $L$ is compact planar and all components of $\partial L - c$ intersect $A$. Let $d$ be one such boundary component. Assume for the moment that $d$ is equal to neither $a$ nor $b$, and denote by $f$ and $g$ the holonomies of the two parts of $A$ separated by $d$. Let the foliated $I$-bundle over $L$ have holonomy $\eta$ on $c$ and $\eta ^{-1}$ on $d$ and identity on all other boundary components. Choose $\eta$ such that $\eta$ is conjugate to $f' \cdot \eta \cdot g'$. Here $f'$ and $ g'$ can be obtained from $f$ and $g$ by trivial Denjoy blow-up of all components of $L \cap (A - d)$. Such a homeomorphism $\eta$ exists by Lemma \ref{concatenation}. 
	
	Now consider the case $d=a$. Choose $\eta$ such that $\eta$ is conjugate to $\eta \cdot g'$, where $g'$ is obtained from $\mu$ after trivial Denjoy blow-up of all components of $\partial L \cap (A - a)$. Such a homeomorphism $\eta$ exists by Lemma \ref{concatenation}. The case $d = b$ is similar.
\end{enumerate}
The partial $I$-bundle replacement along $L$ creates a transverse annulus $B$ above $c$ with holonomy $\eta$. Attach a foliated copy of $\text{annulus} \times I $ with holonomy $\eta$ along $A \cup N(a) \cup B$ to remove $A$ and $B$ from the transverse boundary. This replaces $\tau$ with $\tau - \alpha$.

	Repeat the above procedure for other components of $\tau$.
\end{proof}

\begin{thm}
Let $\mathcal{F}$ be a singular foliation compatible with $(M_1, N, \tau, \mathcal{L}, \gamma)$. There is a sequence of triples $(\mathcal{F}_i, N_i , \tau_i)$ for $1 \leq i \leq n$ starting with $(\mathcal{F}, N , \tau)$ and ending at $(\mathcal{F}_n, N_n, \tau_n = \emptyset)$ such that each term is an $(\mathcal{L}, \gamma)$-extension of the previous term obtained by one of the Operations  \ref{extension-splitting}, \ref{extension-cutting-the-surface}, \ref{extension-embedded-bigon}, \ref{extension-spiralling}, \ref{extension-basic-block}, and \ref{extension-spinning}. 
\label{sequence}
\end{thm}

\begin{proof}
Recall the complexity function $c$ from Definition \ref{train-track-complexity}. In what follows, we abbreviate $c(\mathcal{A}_{N_b(\tau_i)}, \tau_i)$ by $c(\tau_i)$, where $\mathcal{A}_{N_b(\tau_i)}$ is the active subsurface of $\tau_i$ (Definition \ref{active subsurface}). Complete the proof by repeatedly going through \textbf{1)}-\textbf{3)} below.\\

\textbf{1)} Given $\tau_i$, if the train track $\tau_i$ has an embedded bigon: Use Operation \ref{extension-embedded-bigon} to find an $(\mathcal{L}, \gamma)$-extension $(\mathcal{F}_{i+1}, N_{i+1}, \tau_{i+1})$ of $(\mathcal{F}_i, N_i , \tau_i)$ such that  $c(\tau_{i+1}) <  c(\tau_i)$. After doing this finitely many times, there are no embedded bigons left.\\

\textbf{2)} Assume that $\tau_i$ has no embedded bigon, and $\tau_i$ has at least one vertex. By Lemma \ref{reduction} (for the surface $F = \mathcal{A}_{N_b(\tau_i)}$), using a sequence of splitting and collapsing bigons, we may obtain $(\mathcal{F}', \tau')$ from $(\mathcal{F}_i, \tau_i)$ such that at least one of the following holds:
\begin{enumerate}
\item There is a homotopically essential, non $\partial$-parallel, simple closed curve $\delta \subset \mathcal{A}_{N_b(\tau_{i})}$ that is disjoint from $\tau'$.
\item There is a homotopically essential, non $\partial$-parallel, simple closed curve $\gamma \subset \mathcal{A}_{N_b(\tau_{i})}$ intersecting $\tau'$ coherently. 
\item For each component $K$ of $\mathcal{A}_{N_b(\tau')}$, $(K, \tau' \cap K)$ is a generalized basic block.
\end{enumerate}

Set $\tau_{i+1} = \tau'$. By Operations \ref{extension-embedded-bigon} and \ref{extension-splitting} and Lemma \ref{splitting-complexity}, there is an $(\mathcal{L}, \gamma)$-extension $(\mathcal{F}_{i+1}, N_{i+1}, \tau_{i+1})$ of $(\mathcal{F}_i, N_i , \tau_i)$ with $\tau' = \tau_{i+1}$ and $ c(\tau_{i+1}) \leq c(\tau_i)$. 
Now we show how to define $(\mathcal{F}_{i+2}, \tau_{i+2})$ in each of the above three cases: \\

Assume that (1) (respectively (2)) holds. Use Operation \ref{extension-cutting-the-surface} (respectively \ref{extension-spiralling}) to find an $(\mathcal{L}, \gamma)$-extension $(\mathcal{F}_{i+2}, N_{i+2}, \tau_{i+2})$ of $(\mathcal{F}_{i+1}, N_{i+1} , \tau_{i+1})$ such that $c(\tau_{i+2})< c(\tau_{i+1})$.

Now assume that (3) holds. If the generalized basic block has negative Euler characteristic, then use Operation \ref{extension-basic-block} to find an $(\mathcal{L}, \gamma)$-extension $(\mathcal{F}_{i+2}, N_{i+2}, \tau_{i+2})$ of $(\mathcal{F}_{i+1}, N_{i+1} , \tau_{i+1})$ such that $c(\tau_{i+2})< c(\tau_{i+1})$.

At this point, each component of the active subsurface is a union of standard train tracked annuli attached together along their boundaries. If there is any standard train tracked annulus with inward transverse orientation, use Operation \ref{extension-basic-block} to reduce the number of edges of the train track, $c_2$. Then start with a mixed standard train tracked annulus (Figure \ref{standard-train-tracked-annulus}, top-right) whose outward boundary has no edges to the outside, and use Operation \ref{extension-basic-block} to reduce the number of edges of the train track. Now only standard train tracked annuli with outward transverse orientation are left, and we may use Operation \ref{extension-basic-block} to reduce the number of edges of the train track again. This is possible since each such annulus has no edges to the outside. At this point the train track is a union of disjoint simple closed curves. 

\textbf{3)} Assume that $\tau_i$ has no vertices, and is non-empty. Use Operation \ref{extension-spinning} to find an $(\mathcal{L}, \gamma)$-extension $(\mathcal{F}_{i+1}, N_{i+1}, \tau_{i+1})$ of $(\mathcal{F}_i, N_i , \tau_i)$ such that $\tau_{i+1}= \emptyset$.\\

This algorithm terminates since $c(\mathcal{A}_{N_b(\tau)}, \tau)$ is strictly decreasing.
\end{proof}

\begin{proof}[Proof of Theorem \ref{fullymarked}]
	
Let $\mathcal{F}$ be a taut foliation on the closed hyperbolic 3-manifold $M$, and $S$ be a positively oriented fully marked surface. By Proposition \ref{bad annulus}, after possibly replacing $S$ by a homologous surface $S'$ and modifying the foliation $\mathcal{F}$ to $\mathcal{F}'$ using $I$-bundle replacements, we may assume that $\mathcal{F'}|S'$ has the compact-free separatrix property, and $(\mathcal{F'},S')$ has no bad solid tori. Note that the plane fields of $\mathcal{F}$ and $\mathcal{F}'$ are homotopic since they share a common transverse vector field, by Observation \ref{I-bundle-replacement} and Proposition \ref{homotopic}. To simplify the notation, we keep using $(\mathcal{F}, S)$ instead of $(\mathcal{F}', S')$.\\

\textbf{Step 1:} Fix a complete system of coherent transversals $\gamma $ for $(\mathcal{F}, S)$, which exists by Proposition \ref{transversals}. By Proposition \ref{vector-field}, there exists a vector field $\mathcal{L}$ on $M$ that is coherently transverse to both $S$ and $\mathcal{F}$. Pick a collar neighborhood $S \times [0,1]$ of $S = S \times \frac{1}{2}$ in $M$ such that the flow lines of $\mathcal{L}$ induce the vertical fibration inside $S \times [0,1]$. By adjusting $\gamma$, we may assume that the restriction of $\gamma$ to $S \times [0,1]$ is a union of vertical fibers. Let $M_1$ be the closed complement of $S$ in $M$ with the induced foliation $\mathcal{F}_1$. The collar neighborhood of $S$ in $M$ naturally induces a collar neighborhood $\partial M_1 \times [0,1]$ of $\partial M_1$ in $M_1$. Push $S \subset M$ in both sides to obtain a manifold $N_2 \subset M_1$ homeomorphic to $M_1$ whose boundary consists of the union of $S \times \frac{1}{4}$ and $S \times \frac{3}{4}$. Let $\mathcal{F}_2$ be the induced foliation on $N_2$. By enlarging $N_2$ to $N_3$ and $\mathcal{F}_2$ to $\mathcal{F}_3$ we obtain 
\begin{enumerate}
\item $N_2 \subset N_3 \subset M_1$.
\item A train track $\tau  \subset \partial N_3$ and a branched neighborhood $N_b(\tau)$ of $\tau$ such that $N_b(\tau)$ induces a cornered structure on $\partial N_3$, and a transverse orientation on $\partial N_3$ coherent with $N_b(\tau)$.
\item The vector field $\mathcal{L}$ is coherent with the transverse orientation of $\partial N_3$ as well as with the cornered structure of $\partial N_3$. Each vertical fiber of $\partial M_1 \times [0,1]$ intersects $N_3$ in a proper connected interval starting on $\partial M_1 \times 0$, assuming that $\partial M_1$ is identified with $\partial M_1 \times 1$. 
\item If we denote the restriction of $\mathcal{L}$ and $\gamma$ to $M_1$ by $\mathcal{L}_1$ and $\gamma_1$ respectively, and set $\tau_3= \tau$, then $\mathcal{F}_3$ is compatible with $(M_1, N_3, \tau_3, \mathcal{L}_1, \gamma_1)$.
\end{enumerate}
\vskip 8pt

\textbf{Step 2:}  By Theorem \ref{sequence}, there is a sequence of triples $(\mathcal{F}_i, N_i , \tau_i)$ for $3 \leq i \leq n$ with $\tau_n = \emptyset$ such that each term is an $(\mathcal{L}_1, \gamma_1)$-extension of the previous term. In particular $\mathcal{F}_n$ is tangential to $\partial N_n$, and $\partial N_n$ is transverse to the vertical fibration of $\partial M_1 \times [0,1]$.\\

\textbf{Step 3:} Enlarge $N_n$ and $\mathcal{F}_n$ by adding leaves transverse to the vertical fibration of $\partial M_1 \times [0,1]$ to obtain a foliation $\mathcal{G}_0$ of $M_1$. Glue back the two copies of $S$ in $\partial M_1$ (equipped with the foliation $\mathcal{G}_0$) to obtain the foliation $\mathcal{G}$ on $M$. The foliation $\mathcal{G}$ is taut since $\gamma$ is a complete system of coherent transversals for it. The oriented tangent plane fields of $\mathcal{F}$ and $\mathcal{G}$ are homotopic since both of them are transverse to the common vector field $\mathcal{L}$ (Proposition \ref{homotopic}). By construction, $S$ is a union of leaves of $\mathcal{G}$.  
\end{proof}
\begin{remark}
There are analogous results to Theorem \ref{fullymarked} for taut foliations on atoroidal sutured manifolds. In this case, recall that each boundary component of a fully marked surface is either transverse to the foliation or is a leaf.
\end{remark}
\section{Conjectures}
As mentioned in the introduction, we expect Theorem \ref{fullymarked} to fail in general without allowing to change $S$ within its homology class. See Conjecture \ref{generalcase}.


Our proof of the fully marked surface theorem uses the flexibility of $I$-bundle replacement which is a generalisation of \emph{Denjoy blow-up} \cite{denjoy1932courbes}. The operation of Denjoy blow-up can create foliations that are not topologically conjugate to any $C^2$ foliation \cite{denjoy1932courbes}. We expect the analogue of the theorem to be false for $C^2$ taut foliations. 
\begin{conj}
There exists a closed hyperbolic 3-manifold $M$ supporting a $C^2$ taut foliation $\mathcal{F}$ with a fully marked surface $S$, such that there exists no $C^2$ taut foliation $\mathcal{G}$ on $M$ with oriented plane field homotopic to $\mathcal{F}$ such that $S$ is homologous to a union of leaves of $\mathcal{G}$. 
\end{conj}

\bibliographystyle{plain}
\bibliography{foliationreference2}

\end{document}